\documentclass[]{article}
\usepackage{color}
\usepackage{amsmath,bm}
\usepackage{amsfonts}
\usepackage{amssymb}
\usepackage{graphicx}
\usepackage{stmaryrd}
\usepackage{inputenc}
\usepackage{subcaption}
\usepackage{hyperref}

\usepackage{xcolor}

\usepackage{authblk}

\hypersetup{
pdfpagemode={UseOutlines},
bookmarksopen=true,
bookmarksopenlevel=0,
hypertexnames=false,
colorlinks=true,
citecolor=blue,
linkcolor=red,
urlcolor=mdtRed,
pdfstartview={FitV},
unicode,
breaklinks=true,
}

\newtheorem{Theorem}{Theorem}[section]

\newtheorem {Cor}[Theorem]{Corollary}
\newtheorem {pro}[Theorem]{Proposition}

\newtheorem {Lemma}[Theorem]{Lemma}
\newtheorem {rem}[Theorem]{Remark}
\newtheorem {rems}[Theorem]{Remarks}

\newtheorem {com}[Theorem]{Comment}
\newtheorem {coms}[Theorem]{Comments}

\newtheorem {Definition}[Theorem]{Definition}
\newtheorem {exam}[Theorem]{Example}
\newtheorem {exams}[Theorem]{Examples}

\newcommand{\bcom}{\begin{com} \rm }
\newcommand{\ecom}{\end{com}}
\newcommand{\bcoms}{\begin{coms} \rm }
\newcommand{\ecoms}{\end{coms}}
\newcommand {\bdef}{\begin{Definition}}
\newcommand {\edefi}{\end{Definition}}
\newcommand {\bl}{\begin{Lemma}}
\newcommand {\el}{\end{Lemma}}
\newcommand {\bethe}{\begin{Theorem}}
\newcommand {\eethe}{\end{Theorem}}
\newcommand {\bp}{\begin{pro}}
\newcommand {\ep}{\end{pro}}
\newcommand {\bcor}{\begin{Cor}}
\newcommand {\ecor}{\end{Cor}}
\newcommand {\brem }{\begin{rem} \rm }
\newcommand {\erem }{\end{rem}}
\newcommand {\brems }{\begin{rems} \rm }
\newcommand {\erems }{\end{rems}}
\newcommand {\bcorr}{\begin{corr} \rm }
\newcommand {\ecorr}{\hfill $\lhd$ \end{corr}}
\newcommand {\bex}{\begin{exam} \rm }
\newcommand {\eex}{\end{exam}}
\newcommand {\bexs}{\begin{exams} \rm }
\newcommand {\eexs}{\end{exams}}
\newcommand {\benu}{\begin{enumerate}}
\newcommand {\eenu}{\end{enumerate}}
\newcommand {\be}{\begin{equation}}
\newcommand {\ee}{\end{equation}}
\newcommand {\bde}{\begin{displaymath}}
\newcommand {\ede}{\end{displaymath}}
\newcommand {\beq}{\begin{eqnarray*}}
\newcommand {\eeq}{\end{eqnarray*}}
\newcommand {\beqa}{\begin{eqnarray}}
\newcommand {\eeqa}{\end{eqnarray}}

\def \proof {{\sc{Proof:}}~}
\def \finproof {\hfill $ \square$  \\ }

\definecolor{mdtRed}{RGB}{255,0,255}

\def \R{\mathbb R}
\def \E{\mathbb E}

\def \P {\mathbb P}

\def \ind{1\!\!1}

\def \F{{\cal F}}
\def \ff{{\mathbb {F}}}

\usepackage[backend=biber,maxbibnames=7,natbib=true,sorting=nyt,autolang=hyphen]{biblatex}

\begin{document}

\date{}

\title{A Multivariate Generalized Cox Model for Dependent and Simultaneous Default Times}

\author{Djibril Gueye}
\affil{\small{QUANTLABS, Rainbow Partners Group, 92200 Neuilly-sur-Seine, France}}
\author{Alejandra Quintos\thanks{Corresponding author: alejandra.quintos@wisc.edu}\thanks{Supported in part  by the Office of the Vice Chancellor for Research and Graduate Education at the University of Wisconsin-Madison with funding from the Wisconsin Alumni Research Foundation}}
\affil{\small{Department of Statistics, University of Wisconsin-Madison. \\ Madison, WI, 53706}}

\maketitle

\begin{abstract}
We propose a framework for dependent default times where each name defaults when an increasing driver crosses an independent unit exponential threshold. As in Gueye and Jeanblanc (2022), the drivers may jump, so defaults may occur at stopping times of the reference filtration, including predictable ones, and a common jump may cause no, one or several simultaneous defaults. For general drivers, we derive the conditional joint survival law by a backward recursion; it is an explicit product of survival functions if and only if, for every set of names, the predictable reduction of the compensator of their first default time is deterministic. If, moreover, defaults avoid predictable stopping times, the law is a fatal shock law with hazards identified from the drivers; predictable default atoms can rule this out. Examples include L\'evy subordinators and, with random reductions but explicit laws, Liu's model and shot-noise drivers with memory.
\end{abstract}



\section{Introduction}
Modeling dependent default times remains a fundamental challenge in credit risk and insurance, particularly due to the need to capture both simultaneous and non-simultaneous default events. A traditional approach is the Cox model of \citet{lando1998cox}: on a filtered probability space $(\Omega,\mathcal G,\ff,\P)$, the default time $\tau$ is the first time an increasing $\ff$-adapted process $K$, absolutely continuous with respect to the Lebesgue measure, reaches a unit exponential threshold independent of $\ff$. In this model, $\tau$ avoids all $\ff$-stopping times, that is, $\P(\tau= T <\infty)=0$ for every $\ff$-stopping time $T$. This gives tractability, but excludes defaults triggered by observable events, such as major external shocks, which are precisely the events that often cause simultaneous defaults.

The generalized Cox model of \citet{gueye2021generalized} removes this restriction by only requiring $K$ to be increasing, adapted and c\`adl\`ag (or c\`agl\`ad). Default may then occur at a jump time of $K$, which is an $\ff$-stopping time and may be predictable. This class contains, for instance, the hybrid model of \citet{jiao2018modeling} for sovereign risk, in which default can occur at predetermined dates \citep[Section~3.3.2]{gueye2021generalized}. The extension to several default times $\tau^i=\inf\{t\ge0:K^i_t\ge\Theta^i\}$, with independent thresholds $\Theta^i$, is mentioned in \citet[Section~5]{gueye2021generalized} as future work and is the subject of the present paper. Defaults are conditionally independent given $\mathcal F_\infty$, but their dependence comes from the joint randomness of the drivers $K^1,\ldots,K^n$, and common jumps of the drivers allow simultaneous defaults.

First-passage times of increasing processes across independent exponential thresholds are a standard way of generating simultaneous defaults. The common clock is a L\'evy subordinator in the L\'evy-frailty model of \citet{mai2009levy} and the intensity-gamma model of \citet{jS2006}, a subordinator that is infinitely divisible with respect to time in \citet{mai2019subordinators}, a multivariate subordinator in \citet{sun2017marshall}, and a shot-noise process in \citet{sss}; see also \citet{liu2020competing} for an econometric competing risks model. These are generalized Cox models in the sense of the present paper, and several of them serve below as benchmarks.

Other extensions of the Cox framework have been proposed. In the model of \citet{ProtterQuintos2024}, each default time is the minimum of idiosyncratic and common Cox-type components, which yields $\P(\tau^i=\tau^j)>0$ while preserving absolutely continuous compensators; dependence arises from this structural composition rather than from common jumps of the drivers. \citet{chaieb2024two} combine the two approaches in the bivariate setting: the common component of \citet{ProtterQuintos2024} is replaced by a generalized Cox time, so that a common fatal shock, possibly occurring at an $\ff$-stopping time, ends both lifetimes simultaneously. In that construction a shock that is fatal is fatal for both survivors, so a shared shock cannot kill one of two surviving names while sparing the other (see Remark~\ref{rem:Chaieb}).

An independent-threshold generalized Cox construction is also used as the reference model of \citet{coculescu2024default}, under which defaults are conditionally independent given $\ff$ and may occur at $\ff$-stopping times. Their standing assumptions allow each hazard process at most one jump, at a totally inaccessible time, and exclude common jump times across names, so that only one default occurs at a time. Dependence beyond conditional independence is then introduced by a change of probability measure, which creates contagion between names and from defaults onto the environment. The present paper is complementary: we remain under the reference construction, without contagion from defaults, and study drivers with common or partially common jumps, which can generate simultaneous defaults and, for the classes of drivers studied here, explicit joint survival laws.

The contributions of the paper are threefold. First, we give the $\mathcal F_t$-conditional joint survival law of $n$ generalized Cox times for general drivers, through a backward recursion that requires neither deterministic predictable reductions of the compensators nor avoidance of stopping times. We then establish a necessary and sufficient condition for this conditional law to reduce to an explicit ordered product of survival functions: for every nonempty set of names, the $\ff$-predictable reduction of the compensator of their first default time is deterministic. The unconditional product formula holds under a weaker condition, which does not involve the aggregate driver of all names. These results remain valid when defaults occur at predictable dates. Second, when these reductions are deterministic and the default times avoid $\ff$-predictable stopping times, we show that the joint law is that of a fatal shock model with independent shocks, and we identify the cumulative hazard of each shock from the drivers: for a set $B$ of names, it is the dual predictable projection of the conditional probabilities that a jump of $(K^1,\ldots,K^n)$ triggers the defaults of exactly the names in $B$, plus the continuous part of $K^i$ when $B=\{i\}$. Unlike the subordinator constructions of \citet{sun2017marshall,sun2017online}, this requires no independence of increments. Default atoms at predictable dates, however, can take the joint law outside every such model. Third, we show that a common jump of the drivers may cause no default, one default or simultaneous defaults, and that joint laws remain explicit for drivers whose predictable reductions are random, such as shot-noise drivers with memory and the competing risks model of \citet{liu2020competing}. 

The paper is organized as follows. Section~\ref{sect:Known.Results.Definitions} fixes the notation for random times, compensators and their $\ff$-predictable reductions. Section~\ref{sect:Bivariate.Case} recalls the generalized Cox model of \citet{gueye2021generalized} for a single name, shows how avoidance of $\ff$-stopping times and of $\ff$-predictable stopping times can be read off the driver, and characterizes when the predictable reduction of the compensator is deterministic (Lemma~\ref{lem:tractable}). Section~\ref{SecBiv} treats two names: it gives the conditional joint survival function (Theorem~\ref{Thm:Bivariate.Survival}) and the probability of simultaneous default (Proposition~\ref{Prop:tau.1.equal.tau.2}), with three classes of examples: L\'evy subordinators, which recover the Marshall--Olkin laws of \citet{sun2017marshall}; the model of \citet{liu2020competing}; and shot-noise drivers. Section~\ref{sect:Generalization} extends the construction to $n$ names, with the backward recursion (Theorem~\ref{Thm:Multivariate.Recursion}), the ordered product formula (Theorem~\ref{Thm:Multivariate.Survival}), the shock-model representation (Proposition~\ref{prop:continuous.reductions.MO}) and an example outside the shock class (Example~\ref{ex:atomic.signed.mobius}).

\section{Preliminaries on Random Times and Compensators} \label{sect:Known.Results.Definitions}
We consider a probability space $(\Omega,\mathcal G,\P)$ endowed with a random time $\tau$. We fix the filtration, projection, compensator, and multiplicative-decomposition notation used
throughout the paper.

\subsection{Random Times and Default Processes}

Given a random time $\tau$, define the default indicator
$$
A_t:=\ind_{\{\tau\le t\}},\qquad t\ge0.
$$
The process $A$ is increasing and c\`adl\`ag.

\subsection{Projections onto a Reference Filtration}

Let $\mathbb H=(\mathcal H_t)_{t\ge0}$ be a filtration on $(\Omega,\mathcal G)$. The
$\mathbb H$-dual predictable projection $A^{p,\mathbb H}$ is the unique $\mathbb H$-predictable,
integrable, increasing process such that
$$
\E\left[\int_0^\infty Y_s\,dA_s\right]
=
\E\left[\int_0^\infty Y_s\,dA_s^{p,\mathbb H}\right]
$$
for every bounded $\mathbb H$-predictable process $Y$. The dual optional projection
$A^{o,\mathbb H}$ is defined analogously for bounded $\mathbb H$-optional processes. These
projections will be used below to identify the predictable reduction of the default compensator.

\subsection{Multiplicative Decomposition of Supermartingales}

It is well-known that a strictly positive, bounded, c\`adl\`ag $\mathbb H$-supermartingale $Y$
admits a unique normalized multiplicative decomposition
\begin{equation}\label{Multiplicative.Decomposition.Supermartingale}
Y=NC,
\end{equation}
where $N$ is a local $\mathbb H$-martingale with $N_0=1$, and $C$ is a decreasing
$\mathbb H$-predictable process; see, e.g., \cite[Proposition 1.32, p.~15]{aksamit2017enlargement}.
This decomposition will be applied to Az\'ema supermartingales.

\subsection{The Az\'ema Supermartingale and Progressive Enlargement}

Now consider a filtered probability space $(\Omega,\mathcal G,\ff,\P)$, where
$\ff=(\F_t)_{t\ge0}$ is the reference filtration. The Az\'ema supermartingale of $\tau$ is
$$
Z_t:=\P(\tau>t\mid\F_t).
$$
It satisfies $Z_t>0$ on $\{\tau>t\}$ and $Z_{t-}>0$ on $\{\tau\ge t\}$; see
\cite[Lemma 2.14]{aksamit2017enlargement}.

Let $\mathbb G=(\mathcal G_t)_{t\ge0}$ be the progressive enlargement of $\mathbb F$ with $\tau$,
that is, the smallest filtration satisfying the usual conditions, containing $\mathbb F$, and making
$\tau$ a $\mathbb G$-stopping time. Equivalently,
$$
\mathcal G_t=\bigcap_{\varepsilon>0}\left(\mathcal F_{t+\varepsilon}\vee
\sigma(\tau\wedge(t+\varepsilon))\right),
$$
up to the usual completion; see, e.g., \citet{jeulin1978grossissement}.

\subsection{The Compensator and Its \texorpdfstring{$\mathbb F$}{F}-Predictable Reduction}
\label{sect:compensator.reduction}

Since $\tau$ is a $\mathbb G$-stopping time, $A$ is $\mathbb G$-adapted. For the random times
considered in this paper, $\P(\tau=0)=0$, and the compensator of $A$ is the unique
$\mathbb G$-predictable increasing process $\Lambda^{\mathbb G}$, starting from zero, such that
$$
A-\Lambda^{\mathbb G}
$$
is a $\mathbb G$-martingale.

The Jeulin--Yor formula identifies this compensator in terms of the $\mathbb F$-dual predictable
projection of $A$:
\begin{equation}\label{eq:G.compensator.JY}
\Lambda_t^{\mathbb G}
=
\int_{(0,t\wedge\tau]}
\ind_{\{Z_{s-}>0\}}\frac{dA_s^{p,\mathbb F}}{Z_{s-}},
\qquad t\ge0,
\end{equation}
see, e.g., \cite[Proposition 2.15, p.~37]{aksamit2017enlargement}. Since $Z_{s-}>0$ on
$\{s\le\tau\}$, the indicator is harmless on the stochastic interval appearing in
\eqref{eq:G.compensator.JY}, but it is convenient when defining a process on the whole time line.

Accordingly, throughout the paper we choose the canonical $\mathbb F$-predictable reduction of the compensator $\Lambda^{\mathbb G}$
\begin{equation}\label{eq:F.predictable.reduction.compensator}
\Lambda_t
:=
\int_{(0,t]}
\ind_{\{Z_{s-}>0\}}\frac{dA_s^{p,\mathbb F}}{Z_{s-}},
\qquad \Lambda_0=0.
\end{equation}
Then
\begin{equation}\label{eq:G.compensator.reduction}
\Lambda_t^{\mathbb G}=\Lambda_{t\wedge\tau},\qquad t\ge0.
\end{equation}
More generally, a $\mathbb G$-predictable process stopped at $\tau$ admits an
$\mathbb F$-predictable reduction on the pre-default stochastic interval; see, e.g.,
\cite[Proposition 2.11(b), p.~36]{aksamit2017enlargement}. Such a reduction is determined only
through its stopped process and is therefore not unique after $\tau$. Equation
\eqref{eq:F.predictable.reduction.compensator} fixes the canonical version used in this paper.

The terminology is important: $\Lambda$ is not, in general, an $\mathbb F$-compensator of $A$,
because $A$ need not be $\mathbb F$-adapted. The actual compensator is
$\Lambda^{\mathbb G}$ in the enlarged filtration; $\Lambda$ is its $\mathbb F$-predictable reduction.

Finally, we say that $\tau$ \emph{avoids $\mathbb F$-predictable stopping times} if $\P(\tau=T<\infty)=0$ for every $\mathbb F$-predictable stopping time $T$, and that $\tau$ \emph{avoids $\mathbb F$-stopping times} if this holds for every $\mathbb F$-stopping time $T$. The second property holds if and only if $A^{o,\mathbb F}$ is continuous \citep[Theorem~1.43(b)]{aksamit2017enlargement}; the first holds if and only if $A^{p,\mathbb F}$ is continuous, by the same argument with the predictable section theorem \citep[Theorem~1.7(b)]{aksamit2017enlargement} in place of the optional one.
If $\tau$ avoids $\mathbb F$-predictable stopping times, the canonical reduction \eqref{eq:F.predictable.reduction.compensator}, and hence $\Lambda^{\mathbb G}$, is continuous. For the generalized Cox times studied below, Remark~\ref{rem:avoidance} expresses both properties in terms of the driver.

\section{The Generalized Cox Model}\label{sect:Bivariate.Case}

Let us first recall some results from \citet{gueye2021generalized} that will be utilized in the subsequent analysis. Consider the pair $(\Theta, K)$, where $\Theta$ is an exponentially distributed random variable with parameter $1$, independent of the filtration $\ff$, and $K$ is an increasing, c\`adl\`ag, $\ff$-adapted process with $K_0 = 0$. {In this setting, $K$ determines the dynamics of the default time $\tau$, resulting in a generalized Cox model, which} can be defined as (see \citet{gueye2021generalized}):
$$
\tau := \inf \{t \geq 0 : K_t \geq \Theta\}.
$$

The Azéma supermartingale $ Z $ is given by
$$ Z_t = \P(\tau > t \vert \F_{t}) = \P \left( \Theta > K_t \vert \F_t \right) = e^{-K_t} .$$
As established in \citet[Lemma 3.9]{gueye2021generalized} and {in line with the general decomposition result from} equation~\eqref{Multiplicative.Decomposition.Supermartingale}, the Azéma supermartingale $ Z $, being a strictly positive, bounded, c\`adl\`ag $\ff$-supermartingale, admits a unique multiplicative decomposition. To state this decomposition explicitly, we first introduce the processes involved.

\begin{itemize}
    \item $\mathcal{E}(X)$ denotes the Dol\'eans-Dade exponential of a c\`adl\`ag semimartingale $X$ with $X_0=0$ (see \citet[Chapter~II, Theorem~37]{protter2005stochastic}), given by
    $$ \mathcal{E}(X)_t = \exp\left(X_t - \frac{1}{2}\langle X^{(c)}, X^{(c)}\rangle_t\right) \prod_{0 < s \leq t}(1 + \Delta X_s)e^{-\Delta X_s}, $$
    where $X^{(c)}$ denotes the continuous local martingale part of $X$.

    \item The process $A^I$, derived from the jump component of $K$, arises from the canonical decomposition of the special $\ff$-semimartingale
    $$ I_t = \sum_{s \leq t}\left(1 - e^{-\Delta K_s}\right), $$
    which we write as $I = M^I + A^I$, where $M^I$ is a local martingale and $A^I$ is $\ff$-predictable, of finite variation, and null at zero.

        \item
    Let $\Lambda$ be the $\ff$-predictable reduction \eqref{eq:F.predictable.reduction.compensator} of the compensator of $\tau$. Then
    \[
      \Lambda_t=K^c_t+A^I_t,\qquad t\ge0,
    \]
    where $K^c$ denotes the continuous part of $K$; see \citet[Lemma~3.4]{gueye2021generalized}. Indeed, since $K$ is increasing and $K_0=0$,
    \[
      Z_t=e^{-K_t}=1-\int_{(0,t]}Z_{s-}\,dK^c_s+\sum_{0<s\le t}Z_{s-}\big(e^{-\Delta K_s}-1\big)=1-\int_{(0,t]}Z_{s-}\,\big(dK^c_s+dI_s\big),
    \]
    and therefore, writing $I=M^I+A^I$,
    \[
      Z=1-\int_{(0,\cdot]}Z_{s-}\,dM^I_s-\int_{(0,\cdot]}Z_{s-}\,\big(dK^c_s+dA^I_s\big).
    \]
    The second term is a local martingale, and the third is predictable and increasing. Let $Z=M-A^{p,\ff}$ be the Doob--Meyer decomposition of $Z$, where $M$ is a uniformly integrable $\ff$-martingale \citep[see][Proposition~1.46(a)]{aksamit2017enlargement}. By its uniqueness, we get $dA^{p,\ff}=Z_-\,(dK^c+dA^I)$. Since $Z_->0$, \eqref{eq:F.predictable.reduction.compensator} gives $\Lambda=K^c+A^I$.

\end{itemize}

\begin{Lemma}\label{lem:Lambda.Doleans.Behav}
The process $\Lambda$ defined in \eqref{eq:F.predictable.reduction.compensator} is $\ff$-predictable, increasing and finite-valued, and, outside an evanescent set, $0\le\Delta\Lambda<1$. Consequently, $\sum_{0<s\le t}\Delta\Lambda_s\le\Lambda_t<\infty$ for all $t\ge0$ and, writing $\Lambda^c_t:=\Lambda_t-\sum_{0<s\le t}\Delta\Lambda_s$ for the continuous part of $\Lambda$,
\[
  \mathcal E(-\Lambda)_t=e^{-\Lambda^c_t}\prod_{0<s\le t}(1-\Delta\Lambda_s),\qquad t\ge0,
\]
is strictly positive and decreasing, with $\mathcal E(-\Lambda)_0=1$.
\end{Lemma}

\begin{proof}
    The Az\'ema supermartingale $Z=e^{-K}$ is strictly positive and bounded, with Doob--Meyer decomposition $Z=M-A^{p,\ff}$. Since $Z_-=e^{-K_-}>0$, \eqref{eq:F.predictable.reduction.compensator} reads $\Lambda=\frac{1}{Z_-}\cdot A^{p,\ff}$. This is the pathwise Lebesgue--Stieltjes integral of the positive, locally bounded, predictable process $1/Z_-$ against the increasing predictable process $A^{p,\ff}$. Hence $\Lambda$ is predictable, increasing and finite-valued, $\Delta\Lambda=\Delta A^{p,\ff}/Z_-\ge0$, and
    \[
      Z_-(1-\Delta\Lambda)=Z_--\Delta A^{p,\ff}=Z-\Delta M .
    \]

    We adapt the argument in the proof of \citet[Proposition~1.32]{aksamit2017enlargement}. Let $T$ be a predictable time. The variables $Z_{T-}$ and $\Delta A^{p,\ff}_T$ are $\F_{T-}$-measurable, and $\E[\Delta M_T\mid\F_{T-}]=0$ on $\{T<\infty\}$, since the predictable projection of $M$ is $M_-$ \citep[see][Example~1.33]{aksamit2017enlargement}. Taking conditional expectations given $\F_{T-}$ in the previous identity and dividing by $Z_{T-}=e^{-K_{T-}}$ gives, on $\{T<\infty\}$,
    \[
      1-\Delta\Lambda_T=\E\left[e^{-\Delta K_T}\mid\F_{T-}\right]>0 .
    \]
    Since $\Lambda$ is predictable, the set $\{\Delta\Lambda\ge1\}$ is predictable. By the predictable section theorem in \citet[Theorem~1.7(b)]{aksamit2017enlargement}, it is therefore evanescent, that is, $0\le\Delta\Lambda<1$ outside an evanescent set.

    Since $\Lambda$ is increasing, for any $0<s_1<\dots<s_k\le t$ and $s_0:=0$ we have $\Delta\Lambda_{s_i}\le\Lambda_{s_i}-\Lambda_{s_{i-1}}$. Summing, the right-hand side telescopes to at most $\Lambda_t$. Taking the supremum over finite sets of jump times gives $\sum_{0<s\le t}\Delta\Lambda_s\le\Lambda_t<\infty$. Hence $\Lambda^c$ is well defined, and the same argument on $(s,t]$ shows that it is increasing. In particular, $\prod_{0<s\le t}e^{\Delta\Lambda_s}=\exp\{\sum_{0<s\le t}\Delta\Lambda_s\}$ is finite, and $\prod_{0<s\le t}(1-\Delta\Lambda_s)$ converges because its factors lie in $(0,1]$. As $-\Lambda$ has finite variation, its continuous local-martingale part vanishes, and the Dol\'eans-Dade formula above becomes
    \[
      \mathcal E(-\Lambda)_t=e^{-\Lambda_t}\prod_{0<s\le t}(1-\Delta\Lambda_s)\,e^{\Delta\Lambda_s}=e^{-\Lambda^c_t}\prod_{0<s\le t}(1-\Delta\Lambda_s).
    \]

    It remains to show that $\mathcal E(-\Lambda)$ is strictly positive and decreasing. Fix $t\ge0$. Since $\sum_{0<s\le t}\Delta\Lambda_s\le\Lambda_t<\infty$, only finitely many $s\le t$ satisfy $\Delta\Lambda_s\ge\frac12$, and the corresponding factors $1-\Delta\Lambda_s$ are positive because $\Delta\Lambda<1$. For the remaining factors, the inequality $1-x\ge e^{-2x}$ on $[0,\frac12]$ gives
    \[
      \prod_{\substack{0<s\le t\\ \Delta\Lambda_s<1/2}}(1-\Delta\Lambda_s)\ \ge\ \exp\Big\{-2\sum_{0<s\le t}\Delta\Lambda_s\Big\}\ \ge\ e^{-2\Lambda_t}>0 .
    \]
    Hence $\prod_{0<s\le t}(1-\Delta\Lambda_s)>0$, and, as $e^{-\Lambda^c_t}>0$, $\mathcal E(-\Lambda)_t>0$. Moreover, $t\mapsto e^{-\Lambda^c_t}$ is positive and decreasing because $\Lambda^c$ is increasing. Likewise, $t\mapsto\prod_{0<s\le t}(1-\Delta\Lambda_s)$ is positive and decreasing because each factor lies in $(0,1]$. Their product $\mathcal E(-\Lambda)$ is therefore decreasing. Finally, $\mathcal E(-\Lambda)_0=1$ since $\Lambda_0=0$.
    \finproof
\end{proof}

Since $K^c$ is continuous, $\Delta\Lambda=\Delta A^I$, so Lemma~\ref{lem:Lambda.Doleans.Behav} gives $1-\Delta A^I>0$. The process
\[
  Y_t=\int_0^t\frac{1}{Z_{s-}\,(1-\Delta A^I_s)}\,dM_s,\qquad t\ge0,
\]
is well defined, because its integrand is predictable and locally bounded: $1/Z_-=e^{K_-}$ is left-continuous and adapted, and $\mathcal E(-\Lambda)_t=\mathcal E(-\Lambda)_{t-}(1-\Delta\Lambda_t)\le 1-\Delta\Lambda_t$ gives $1/(1-\Delta A^I)\le 1/\mathcal E(-\Lambda)$, which is predictable, increasing and finite-valued by Lemma~\ref{lem:Lambda.Doleans.Behav}. By \citet[Lemma~3.9]{gueye2021generalized}, $\eta:=\mathcal E(Y)$ is the local-martingale factor in the normalized multiplicative decomposition of $Z$, that is,

\be \label{DM1}
Z_t = \eta_t\,\mathcal{E}(-\Lambda)_t
    = \eta_t\, e^{-\Lambda^c_t}\prod_{0<s\le t}(1-\Delta\Lambda_s),
    \qquad t\ge 0,
\ee
where the second equality is Lemma~\ref{lem:Lambda.Doleans.Behav}.

 The decomposition \eqref{DM1} separates the dynamics of $Z$
into a local martingale and a predictable finite-variation component.

\par\medskip

In the general normalized multiplicative decomposition, $\eta=\mathcal E(Y)$ is a strictly positive
local $\ff$-martingale with $\eta_0=1$, while $\Lambda$ is the $\ff$-predictable reduction of the
compensator of $\tau$. The true-martingale property of $\eta$ needed later is established separately under
conditions that make the predictable factor deterministic.

\brem\label{rem:predictable}
The processes $\Lambda$ and $K$ differ as soon as $K$ jumps, even when $K$ is $\mathbb F$-predictable. Indeed, suppose $K$ is $\mathbb F$-predictable. Since $\Delta K=K-K_-$ is then predictable,
$I_t=\sum_{0<s\le t}(1-e^{-\Delta K_s})$ is a predictable increasing process with jumps in $[0,1)$,
hence its own compensator: $M^I\equiv 0$ and $A^I=I$. Consequently
\[
  \Lambda_t = K^c_t+\sum_{0<s\le t}\bigl(1-e^{-\Delta K_s}\bigr).
\]
Since $K^c$ is continuous, $\Delta\Lambda_s=1-e^{-\Delta K_s}<\Delta K_s$ whenever $\Delta K_s>0$. Thus $\Lambda$ and $K$ have the same continuous part but different jumps; in
particular, $\Lambda=K$ if and only if $K$ is continuous.

Although $\Lambda\neq K$ when $K$ jumps, $\mathcal E(-\Lambda)=e^{-K}$: since $\Lambda^c=K^c$ and $1-\Delta\Lambda_s=e^{-\Delta K_s}$, Lemma~\ref{lem:Lambda.Doleans.Behav} gives
\[
  \mathcal E(-\Lambda)_t
  = e^{-K^c_t}\prod_{0<s\le t}(1-\Delta\Lambda_s)
  = e^{-K^c_t}\prod_{0<s\le t}e^{-\Delta K_s}
  = e^{-K_t}.
\]
Hence $\eta\equiv1$ in \eqref{DM1}, and $K$ is recovered from $\Lambda$ as
$K=-\log\mathcal E(-\Lambda)$, not as $\Lambda$ itself. Conversely, if $\eta\equiv1$, then \eqref{DM1} gives $e^{-K}=\mathcal E(-\Lambda)$, which is predictable because $\Lambda$ is; hence $K$ is predictable. Thus $\eta\equiv1$ if and only if $K$ is $\mathbb F$-predictable.
\erem

\begin{rem}\label{rem:avoidance}
For the generalized Cox time $\tau$, both avoidance properties of Section~\ref{sect:compensator.reduction} can be read off $K$. Since $\P(\tau\le t\mid\mathcal F_\infty)=1-e^{-K_t}$ is $\mathbb F$-adapted, $A^{o,\mathbb F}=1-e^{-K}$; and since $Z_->0$, $\Lambda$ is continuous if and only if $A^{p,\mathbb F}$ is. Moreover, for every $\mathbb F$-stopping time $T$,
\[
  \P(\tau=T<\infty)=\E\bigl[e^{-K_{T-}}\bigl(1-e^{-\Delta K_T}\bigr)\ind_{\{T<\infty\}}\bigr].
\]
Hence:
\begin{enumerate}
  \item[(a)] $\tau$ avoids $\mathbb F$-stopping times if and only if $K$ is continuous. In this case $\Lambda=K$; it contains the classical Cox model, in which $K$ is absolutely continuous.
  \item[(b)] $\tau$ avoids $\mathbb F$-predictable stopping times if and only if $\Lambda$ is continuous, or equivalently, if and only if $\Delta K_T=0$ a.s.\ on $\{T<\infty\}$ for every $\mathbb F$-predictable stopping time $T$. Then $\mathcal E(-\Lambda)=e^{-\Lambda}$, but $\Lambda\neq K$ as soon as $K$ jumps. The L\'evy subordinators of Section~\ref{Subsect:LevySubordinators}, the shot-noise drivers of Section~\ref{sn}, and the reference model of \citet{coculescu2024default} in Remark~\ref{rem:CV} are of this type.
  \item[(c)] If $\P(\Delta K_T>0,\ T<\infty)>0$ for some $\mathbb F$-predictable stopping time $T$, for instance a deterministic date as in Section~\ref{subsect:Liu}, then $\P(\Delta\Lambda_T>0)>0$ and $\P(\tau=T<\infty)>0$, that is, default occurs exactly at the predictable time $T$ with positive probability.
\end{enumerate}
We impose neither (a) nor (b) in general: continuity of the reductions is used only in Proposition~\ref{prop:continuous.reductions.MO}, and none of the other results below require $\Lambda$ to be continuous.
\end{rem}

For $0\le s\le u$ we define the survival function and the conditional survival factor of $\tau$ by
\[
  S(t):=\P(\tau>t),
  \qquad
  R(s,u):=\frac{\P(\tau>u\mid\mathcal F_s)}{\P(\tau>s\mid\mathcal F_s)}
         =\frac{\P(\tau>u\mid\mathcal F_s)}{Z_s},
\]
which is well defined since $Z_s=e^{-K_s}>0$. Since $\P(\tau>u\mid\mathcal F_\infty)=e^{-K_u}$ and $Z_s$ is $\mathcal F_s$-measurable,
\begin{equation} \label{Prop:R}
  S(t)=\E\left[e^{-K_t}\right]>0,
  \qquad
  R(s,u)=\E\left[e^{-(K_u-K_s)}\,\big|\,\mathcal F_s\right].
\end{equation}
Equivalently, multiplying the second identity by the $\mathcal F_s$-measurable factor $e^{-K_s}$,
\begin{equation}\label{eq:R.identity}
  \E\bigl[e^{-K_u}\mid\mathcal F_s\bigr]=e^{-K_s}\,R(s,u),\qquad 0\le s\le u .
\end{equation}
Since $\{\tau>s\}$ need not be $\mathcal F_s$-measurable, $R(s,u)$ is not a conditional probability given $\mathcal F_s$. It is, however, the survival probability up to time $u$ given the progressively enlarged information $\mathcal G_s$: by \citet[Lemma~2.9]{aksamit2017enlargement},
\[
  \P(\tau>u\mid\mathcal G_s)=\ind_{\{\tau>s\}}\,R(s,u),\qquad 0\le s\le u .
\]

Some of the closed-form results below require $\Lambda$ to be deterministic. Since $\Lambda=K^c+A^I$, a simple sufficient condition is the following.
\begin{enumerate}
    \item[({H1}):] The continuous part $K^c$ of $K$ is  deterministic.
    \item[(H2):] The process $A^I$ is  deterministic.
\end{enumerate}
These conditions are not necessary, since only the sum $K^c+A^I$ needs to be deterministic.

\begin{Lemma}\label{lem:tractable}
    The following are equivalent:
    \begin{enumerate}
        \item[(i)] $\Lambda$ is deterministic, or equivalently, $\mathcal E(-\Lambda)$ is deterministic;
        \item[(ii)] the process $e^{-K_t}/S(t)$, $t\ge0$, is an $\mathbb F$-martingale;
        \item[(iii)] $R(s,u)$ is deterministic for all $0\le s\le u$.
    \end{enumerate}
    Moreover, as the proof shows, when these conditions hold we have
    \[
      \mathcal E(-\Lambda)=S,\qquad \eta=\frac{e^{-K}}{S},\qquad
      R(s,u)=\frac{S(u)}{S(s)},\qquad
      \Lambda_t=-\int_{(0,t]}\frac{dS(u)}{S(u-)},
    \]
   so that $\Lambda$ is the cumulative hazard function of $\tau$, with $\Delta\Lambda_t=\P(\tau=t\mid\tau\ge t)$. In general, $\Lambda\neq-\log S$: the two coincide if and only if $\tau$ has no atoms. Finally, (H1)--(H2) imply (i).
\end{Lemma}

\begin{proof}

    \emph{The two forms of (i).} If $\Lambda$ is deterministic, so is $\mathcal E(-\Lambda)$, by the product formula of Lemma~\ref{lem:Lambda.Doleans.Behav}. Conversely, by definition, the Dol\'eans-Dade exponential solves $\mathcal E(-\Lambda)_t=1-\int_{(0,t]}\mathcal E(-\Lambda)_{u-}\,d\Lambda_u$ (see \citet[Chapter~II, Theorem~37]{protter2005stochastic}). Since $\mathcal E(-\Lambda)$ is decreasing and strictly positive by Lemma~\ref{lem:Lambda.Doleans.Behav}, we have $\mathcal E(-\Lambda)_{u-}\ge\mathcal E(-\Lambda)_u>0$, and therefore
    \begin{equation}\label{eq:Lambda.from.E}
      \Lambda_t=-\int_{(0,t]}\frac{d\mathcal E(-\Lambda)_u}{\mathcal E(-\Lambda)_{u-}},
    \end{equation}
    which is deterministic if $\mathcal E(-\Lambda)$ is. In the rest of the proof we use (i) in the form ``$\mathcal E(-\Lambda)$ is deterministic''.

    \emph{(i)$\Rightarrow$(ii).} Assume that $\mathcal E(-\Lambda)$ is
    deterministic. By Lemma \ref{lem:Lambda.Doleans.Behav}, it is strictly positive and decreasing. Fix $T<\infty$. As $0<e^{-K_t}\le1$, for $0\le t\le T$,
    \[
      0<\eta_t=\frac{e^{-K_t}}{\mathcal E(-\Lambda)_t}
      \le\frac{1}{\mathcal E(-\Lambda)_t}\le\frac{1}{\mathcal E(-\Lambda)_T}.
    \]
    Thus $\eta$ is a bounded local martingale on $[0,T]$, hence a true martingale,
    and $\E[\eta_t]=\eta_0=1$. Taking expectations in
    $e^{-K_t}=\eta_t\,\mathcal E(-\Lambda)_t$ and using that
    $\mathcal E(-\Lambda)_t$ is deterministic gives
    $S(t)=\mathcal E(-\Lambda)_t\,\E[\eta_t]=\mathcal E(-\Lambda)_t$. Therefore
    $\eta=e^{-K}/S$, and this process is a martingale, as just shown. This is
    (ii).

    \emph{(ii)$\Rightarrow$(i).} The function $S$ is deterministic, hence
    predictable; it satisfies $S(0)=1$, is strictly positive, and is decreasing
    because $K$ is increasing. By (ii), $e^{-K}/S$ is a martingale starting from
    $1$. Hence
    \[
      e^{-K}=\frac{e^{-K}}{S}\cdot S
    \]
    is a normalized multiplicative decomposition of $Z=e^{-K}$. By uniqueness of
    the multiplicative decomposition, i.e., equation~\eqref{Multiplicative.Decomposition.Supermartingale},
    $\eta=e^{-K}/S$ and $\mathcal E(-\Lambda)=S$, which is deterministic.

    \emph{(ii)$\Rightarrow$(iii).} By the martingale property in (ii),
    $\E[e^{-K_u}/S(u)\mid\mathcal F_s]=e^{-K_s}/S(s)$. Multiplying by $S(u)$ and
    comparing with {equation~\eqref{eq:R.identity}} gives $e^{-K_s}R(s,u)=e^{-K_s}S(u)/S(s)$. Since
    $e^{-K_s}>0$, $R(s,u)=S(u)/S(s)$, which is deterministic.

    \emph{(iii)$\Rightarrow$(ii).} If $R(s,u)$ is deterministic, taking
    expectations on both sides of {equation~\eqref{eq:R.identity}} gives
    $S(u)=R(s,u)\,\E[e^{-K_s}]=R(s,u)\,S(s)$, that is, $R(s,u)=S(u)/S(s)$.
    Substituting this into {equation~\eqref{eq:R.identity}} and dividing by $S(u)>0$ yields
    \[
      \E\left[\frac{e^{-K_u}}{S(u)}\,\Big|\,\mathcal F_s\right]=\frac{e^{-K_s}}{S(s)},
      \qquad 0\le s\le u .
    \]
    Since $e^{-K}/S$ is adapted and bounded by $1/S(T)$ on $[0,T]$, it is a
    martingale.

    \emph{Identification of $\Lambda$.} Under (i)--(iii), $\mathcal E(-\Lambda)=S$, and equation~\eqref{eq:Lambda.from.E} gives $\Lambda_t=-\int_{(0,t]}dS(u)/S(u-)$, the cumulative hazard function of the law of $\tau$. In particular, $\Delta\Lambda_t=\bigl(S(t-)-S(t)\bigr)/S(t-)=\P(\tau=t)/\P(\tau\ge t)$. Moreover, by Lemma~\ref{lem:Lambda.Doleans.Behav},
    \[
      -\log S(t)=\Lambda^c_t-\sum_{0<s\le t}\log(1-\Delta\Lambda_s)\ \ge\ \Lambda_t,
    \]
    with equality for all $t$ if and only if $\Lambda$, equivalently $S$, has no jumps, that is, if and only if $\tau$ has no atoms, since $-\log(1-x)>x$ for $x\in(0,1)$.

    Finally, under (H1)--(H2), $\Lambda=K^c+A^I$ is deterministic, which is (i).
\finproof

\end{proof}

\section{The Bivariate Case}\label{SecBiv}
\brem \textbf{(Notation)}
    To avoid ambiguity, we adopt the convention that superscripts referring to sets of indices are enclosed in braces. For example, we write $\tau^{\{1,2\}}$, $K^{\{1,2\}}$, or {$S^{\{1,2\}}$} when the superscript involves more than a singleton. When the index consists of a single element, we omit the brackets and simply write, for instance, $\tau^1$ or $K^2$. This convention will be used consistently throughout the remainder of the paper.
\erem

We now apply the construction of Section~\ref{sect:Bivariate.Case} to two names. For $i = 1, 2$, define the default time
$$\tau^{i} := \inf \left\{t \geq 0 \,:\, K^{i}_t\geq \Theta^{i}\right\},$$
where each $K^{i}$ is an $\ff$-adapted, c\`adl\`ag, increasing process with $K^{i}_0=0$, and $\Theta^1,\Theta^2$ are independent unit exponential random variables such that the pair $(\Theta^1,\Theta^2)$ is independent of $\mathcal F_\infty$. The processes $K^1, K^2$ are allowed to be dependent,  which may induce dependence among the default times $\tau^1, \tau^2$. In fact, in Proposition \ref{Prop:tau.1.equal.tau.2}, we will show that the two default times coincide with positive probability if and only if $K^1$ and $K^2$ jump simultaneously with positive probability.

The corresponding Az\'ema supermartingale is
$$
Z_t^i:=\P(\tau^i>t\mid\F_t)=e^{-K_t^i},
$$
and its normalized predictable multiplicative decomposition is

\begin{equation}\label{DMi}
    Z_t^i = \eta_t^i\,\mathcal E(-\Lambda^i)_t.
\end{equation}
At this stage $\eta^i$ is a positive local martingale and $\Lambda^i$ is not assumed deterministic. If the equivalent conditions of Lemma \ref{lem:tractable} hold for $K^i$, then $\eta^i=e^{-K^i}/S^i$ is a true martingale.

\par\medskip

Define
$$
K_t^{\{1,2\}}:=K_t^1+K_t^2,
\qquad
\tau^{\{1,2\}}:=\tau^1\wedge\tau^2.
$$
Then, directly from the construction,
$$
Z_t^{\{1,2\}}
:=
\P(\tau^{\{1,2\}}>t\mid\F_t)
=
\P(\tau^1>t,\tau^2>t\mid\F_t)
=
e^{-K_t^{\{1,2\}}}.
$$
We write its normalized predictable multiplicative decomposition as
\begin{equation}\label{DM_3}
    Z^{\{1,2\}}_t=\eta^{\{1,2\}}_t\,\mathcal E\left(-\Lambda^{\{1,2\}}\right)_t .
\end{equation}
Again, no determinism or avoidance assumption is required for this decomposition.

Since $\P\left(\tau^{\{1,2\}}>t\mid\mathcal F_\infty\right)=e^{-K^{\{1,2\}}_t}$, the results above, in particular Lemmas~\ref{lem:Lambda.Doleans.Behav} and~\ref{lem:tractable}, apply to $\tau^{\{1,2\}}$ with $K$ replaced by $K^{\{1,2\}}$. Accordingly, for $J\in\left\{\{1\},\{2\},\{1,2\}\right\}$ we write $S^J$, $R^J$, $\Lambda^J$ and $\eta^J$ for the objects associated with $\tau^J$ and $K^J$.

The general bivariate identity below requires neither (H1)--(H2) nor any avoidance condition. For its explicit unconditional form we assume that the equivalent conditions of Lemma~\ref{lem:tractable} hold for $K^1$ and $K^2$; for its dynamic form we also assume them for $K^{\{1,2\}}$. By the last assertion of that lemma, (H1)--(H2) imposed on each $K^J$ provide one sufficient criterion. For $J=\{1,2\}$, (H1) follows automatically from the marginal versions because
$K^{\{1,2\},c}=K^{1,c}+K^{2,c}$, whereas (H2) is an additional aggregate restriction. Indeed,
$$
I_t^{\{1,2\}}
=
I_t^1+I_t^2
-
\sum_{0< s\le t}
(1-e^{-\Delta K_s^1})(1-e^{-\Delta K_s^2}),
$$
so determinism of the predictable finite-variation parts of $I^1$ and $I^2$ does not by itself
imply the corresponding determinism for $I^{\{1,2\}}$. For instance, let $X$ be a non-degenerate $[0,1]$-valued random variable, let $\mathbb F$ be generated by $X$ and $(N^a,N^b,N^c)$, and, conditionally on $X$, let $N^a,N^b,N^c$ be independent Poisson processes with intensities $\lambda(1-X)$, $\lambda(1-X)$ and $\lambda X$. Then $K^1:=N^a+N^c$ and $K^2:=N^b+N^c$ are Poisson processes with intensity $\lambda$ and increments independent of the past, so the equivalent conditions of Lemma~\ref{lem:tractable} hold for $K^1$ and $K^2$, whereas $\Lambda^{\{1,2\}}_t=t\bigl(2\lambda(1-X)(1-e^{-1})+\lambda X(1-e^{-2})\bigr)$ is random. As Theorem~\ref{Thm:Bivariate.Survival} shows, the aggregate condition is needed only for the dynamic form \eqref{eq:biv.general.deterministic}, not for the unconditional law.

\begin{rem}\label{rem:Chaieb}
\citet{chaieb2024two} use the generalized Cox model differently, for the joint lifetimes of a couple. In their construction, the lifetimes are $\xi^1\wedge \xi^0$ and $\xi^2\wedge\xi^0$, where $\xi^1,\xi^2$ are classical Cox times with absolutely continuous hazard processes, $\xi^0$ is a generalized Cox time driven by an increasing process $K^0$ that jumps at $\mathbb F$-stopping times, and the three thresholds are independent unit exponentials, independent of $\mathcal F_\infty$. This extends the construction of \citet{ProtterQuintos2024}, in which $\xi^0$ is a classical Cox time and hence avoids all $\mathbb F$-stopping times, by letting $\xi^0$ be a generalized Cox time. Conditionally on $\mathcal F_\infty$, the two lifetimes are dependent through the shared time $\xi^0$, and a fatal shock kills both lives if both are alive; in particular, at a jump time of $K^0$ one life can end without the other only if the other has already ended. In the present construction, $\tau^1$ and $\tau^2$ are conditionally independent given $\mathcal F_\infty$, and the dependence comes from common jumps of $K^1$ and $K^2$, each of which can trigger no default, one default, or two simultaneous defaults. The two constructions differ in this conditional mechanism, not necessarily in the resulting joint laws: for instance, both produce the classical bivariate Marshall--Olkin law, in \citet{chaieb2024two} when $\xi^1,\xi^2$ have constant hazard rates and $K^0$ is a L\'evy subordinator, and here in Section~\ref{Subsect:LevySubordinators}.
\end{rem}

The following identity is obtained by successive conditioning. Its role is to isolate the conditional increment factor that remains available when the predictable multiplicative factors are stochastic.

\begin{Theorem}\label{Thm:Bivariate.Survival}
For $i=1,2$, let $R^i$ be the conditional survival factor of $\tau^i$.
Then, for $0\le t\le\min(t_1,t_2)$,
\begin{equation}\label{eq:biv.general.R}
    \P(\tau^1>t_1,\tau^2>t_2\mid\F_t)
    =
    \begin{cases}
        \displaystyle
        \E\left[
        e^{-K_{t_1}^{\{1,2\}}}
        R^2(t_1,t_2)
        \,\Bigm|\,
        \F_t
        \right],
        & t_1\le t_2,
        \\[1.1em]
        \displaystyle
        \E\left[
        e^{-K_{t_2}^{\{1,2\}}}
        R^1(t_2,t_1)
        \,\Bigm|\,
        \F_t
        \right],
        & t_2<t_1.
    \end{cases}
\end{equation}
No assumption (H1)--(H2) and no avoidance condition are required for
\eqref{eq:biv.general.R}.

If, in addition, the equivalent conditions of Lemma~\ref{lem:tractable} hold for $K^1$ and $K^2$, then
\begin{equation}\label{eq:R.deterministic.D}
    R^i(s,u)=\frac{S^i(u)}{S^i(s)}
\end{equation}
and, for $0\le t\le t_1\le t_2$,
\[
  \P(\tau^1>t_1,\tau^2>t_2\mid\F_t)
  =e^{-K^{\{1,2\}}_t}R^{\{1,2\}}(t,t_1)\,\frac{S^2(t_2)}{S^2(t_1)},
\]
with the symmetric expression for $t_2<t_1$. If the equivalent conditions of Lemma~\ref{lem:tractable} also hold for $K^{\{1,2\}}$, then $e^{-K^{\{1,2\}}_t}R^{\{1,2\}}(t,t_1)=\eta^{\{1,2\}}_t S^{\{1,2\}}(t_1)$, with $\eta^{\{1,2\}}_t=\frac{e^{-K^{\{1,2\}}_t}}{S^{\{1,2\}}(t)}$, and we have
\begin{equation} \label{eq:biv.general.deterministic}
    \P(\tau^1>t_1,\tau^2>t_2\mid\F_t)
    =
    \begin{cases}
        \displaystyle
        \eta^{\{1,2\}}_t\,S^{\{1,2\}}\left(t_1\right)\,\frac{S^2\left(t_2\right)}{S^2(t_1)},
        & t_1\le t_2,
        \\[1.1em]
        \displaystyle
        \eta^{\{1,2\}}_t\,S^{\{1,2\}}(t_2)\,\frac{S^1(t_1)}{S^1(t_2)},
        & t_2<t_1.
    \end{cases}
\end{equation}
Finally, the unconditional joint survival function \eqref{eq:biv.unconditional} below holds as soon as the conditions of Lemma~\ref{lem:tractable} hold for $K^1$ and $K^2$:
\begin{equation} \label{eq:biv.unconditional}
    \P(\tau^1>t_1,\tau^2>t_2)
    =
    \begin{cases}
        \displaystyle
        S^{\{1,2\}}(t_1)\,\frac{S^2(t_2)}{S^2(t_1)},
        & t_1\le t_2,
        \\[1.1em]
        \displaystyle
        S^{\{1,2\}}(t_2)\,\frac{S^1(t_1)}{S^1(t_2)},
        & t_2<t_1.
    \end{cases}
\end{equation}

\end{Theorem}

\proof
Suppose $t_1\le t_2$. Since $(\Theta^1,\Theta^2)$ is independent of $\mathcal F_\infty$ and $\Theta^1,\Theta^2$ are independent,
$$
\P(\tau^1>t_1,\tau^2>t_2\mid\F_t)
=
\E\left[e^{-K_{t_1}^1-K_{t_2}^2}\mid\F_t\right].
$$
Since
$$
K_{t_1}^1+K_{t_2}^2
=
K_{t_1}^{\{1,2\}}
+
(K_{t_2}^2-K_{t_1}^2),
$$
using the tower property at time $t_1$ and equation~\eqref{Prop:R}, we get $\E\left[e^{-(K^2_{t_2}-K^2_{t_1})}\mid\mathcal F_{t_1}\right]=R^2(t_1,t_2)$. The other ordering is symmetric.

Under the additional assumption on $K^1$ and $K^2$, Lemma~\ref{lem:tractable} gives $R^2(t_1,t_2)=S^2(t_2)/S^2(t_1)$, which is deterministic and can be taken out of the conditional expectation in \eqref{eq:biv.general.R}. By equation~\eqref{eq:R.identity} applied to $K^{\{1,2\}}$, $\E\bigl[e^{-K^{\{1,2\}}_{t_1}}\mid\mathcal F_t\bigr]=e^{-K^{\{1,2\}}_t}R^{\{1,2\}}(t,t_1)$, which gives the first display. If the conditions of Lemma~\ref{lem:tractable} also hold for $K^{\{1,2\}}$, then $R^{\{1,2\}}(t,t_1)=S^{\{1,2\}}(t_1)/S^{\{1,2\}}(t)$, which gives \eqref{eq:biv.general.deterministic}. Taking expectations in \eqref{eq:biv.general.R}, with $R^2(t_1,t_2)$ deterministic, gives $\P(\tau^1>t_1,\tau^2>t_2)=S^{\{1,2\}}(t_1)\,S^2(t_2)/S^2(t_1)$, which uses no condition on $K^{\{1,2\}}$. The case $t_2<t_1$ is symmetric.

\finproof

\brem\label{rem:SML}
By analogy with the shock rates of \citet[eq.~(9)]{sun2017marshall} and the cumulative intensities $A^1,A^2,A^3$ of \citet[Theorem~2.1]{ProtterQuintos2024},
define, for $t\ge0$,
\[
  \Gamma^1_t:=\log S^2(t)-\log S^{\{1,2\}}(t),\qquad
  \Gamma^2_t:=\log S^1(t)-\log S^{\{1,2\}}(t),
\]
\[
  \Gamma^{\{1,2\}}_t:=\log S^{\{1,2\}}(t)-\log S^1(t)-\log S^2(t).
\]
These coordinates decompose the aggregate survival functions,
\[
  -\log S^1=\Gamma^1+\Gamma^{\{1,2\}},\qquad
  -\log S^2=\Gamma^2+\Gamma^{\{1,2\}},\qquad
  -\log S^{\{1,2\}}=\Gamma^1+\Gamma^2+\Gamma^{\{1,2\}},
\]
and, whenever equation \eqref{eq:biv.general.deterministic} holds,
\[
  \P(\tau^1>t_1,\tau^2>t_2\mid\mathcal F_t)
  =\eta^{\{1,2\}}_t\exp\Bigl\{-\Gamma^1_{t_1}-\Gamma^2_{t_2}
   -\Gamma^{\{1,2\}}_{t_1\vee t_2}\Bigr\}.
\]
In probabilistic terms, $\Gamma^1_t=-\log\P(\tau^1>t\mid\tau^2>t)$, and symmetrically for
$\Gamma^2$. Moreover, whenever the unconditional formula of Theorem~\ref{Thm:Bivariate.Survival} holds,
\[
  \P(\tau^1>t_1,\tau^2>t_2)=e^{\Gamma^{\{1,2\}}_{t_1\wedge t_2}}\,S^1(t_1)\,S^2(t_2).
\]
Hence $\Gamma^{\{1,2\}}$ measures the departure from independence:
$(\tau^1,\tau^2)$ is positively quadrant dependent, i.e.\
$\P(\tau^1>t_1,\tau^2>t_2)\ge\P(\tau^1>t_1)\P(\tau^2>t_2)$ for all $t_1,t_2$, if and only if
$\Gamma^{\{1,2\}}\ge0$, and $\tau^1,\tau^2$ are independent if and only if $\Gamma^{\{1,2\}}\equiv0$.

If, in addition, $\Gamma^1$, $\Gamma^2$ and $\Gamma^{\{1,2\}}$ are nondecreasing, the unconditional law of $(\tau^1,\tau^2)$ is that of $\bigl(E_1\wedge E_{\{1,2\}},\,E_2\wedge E_{\{1,2\}}\bigr)$ for independent shock times $E_1,E_2,E_{\{1,2\}}$ with values in $(0,\infty]$ and
$\P(E_B>t)=e^{-\Gamma^B_t}$ for $B\in\{\{1\},\{2\},\{1,2\}\}$, that is, a law of time-inhomogeneous Marshall--Olkin type, which is the classical Marshall--Olkin law \cite{marshall1967multivariate} when the $\Gamma^B$ are linear. Thus the superscript of $\Gamma^{\{1,2\}}$ refers to the shock that kills both names, whereas $S^{\{1,2\}}$ refers to the first default among them, whose minus log-survival function $\Gamma^1+\Gamma^2+\Gamma^{\{1,2\}}$ involves all three shocks. The multivariate versions of these coordinates and of the shock representation are given in \eqref{eq:Mobius.Gamma} and Remark~\ref{rem:MO.interpretation}, respectively.
\erem

\bp \label{Prop:tau.1.equal.tau.2}
The two default times coincide with probability
\[
  \P\left(\tau^1=\tau^2<\infty\right)
  =\E\Bigl[\sum_{0<s<\infty}\Bigl(e^{-K^1_{s-}}-e^{-K^1_s}\Bigr)\Bigl(e^{-K^2_{s-}}-e^{-K^2_s}\Bigr)\Bigr].
\]
In particular, $\P(\tau^1=\tau^2<\infty)>0$ if and only if, with positive probability,
$K^1$ and $K^2$ jump at the same time.
\ep

\proof
Since $(\Theta^1,\Theta^2)$ is independent of $\F_\infty$, we may compute conditionally on
$\F_\infty$ with the paths of $K^1,K^2$ frozen. Since $\{\tau^j>t\}=\{K^j_t<\Theta^j\}$, the conditional law $\mu^j$ of $\tau^j$ on $(0,\infty]$
satisfies $\mu^j((t,\infty])=e^{-K^j_t}$; its atoms in $(0,\infty)$ are the jump times of $K^j$,
with $\mu^j(\{s\})=e^{-K^j_{s-}}-e^{-K^j_s}$. As $\Theta^1$ and $\Theta^2$ are independent, the
conditional joint law is $\mu^1\otimes\mu^2$, and by Fubini's theorem
\[
  \P(\tau^1=\tau^2<\infty\mid\F_\infty)=\int_{(0,\infty)}\mu^2(\{s\})\,\mu^1(ds)
  =\sum_{0<s<\infty}\mu^1(\{s\})\,\mu^2(\{s\}),
\]
since $s\mapsto\mu^2(\{s\})$ vanishes off the countable set of jump times of $K^2$. Taking
expectations gives the formula. Each term equals
$e^{-K^{\{1,2\}}_{s-}}(1-e^{-\Delta K^1_s})(1-e^{-\Delta K^2_s})$, which is positive exactly when
both processes jump at $s$; this gives the last assertion.
\finproof

If $(T_i)_{i\ge1}$ are $\ff$-stopping times with pairwise disjoint graphs covering the common jump times, which always exist, the sum is $\sum_i(e^{-K^1_{T_i-}}-e^{-K^1_{T_i}})(e^{-K^2_{T_i-}}-e^{-K^2_{T_i}})\ind_{\{T_i<\infty\}}$; no restriction to predictable or totally inaccessible times is needed.

\subsection{Examples}  \label{ExplBiv}

The examples below illustrate the cases distinguished in Remark~\ref{rem:avoidance} and the two parts of Theorem~\ref{Thm:Bivariate.Survival}. For L\'evy subordinators and their compound Poisson special case (Section~\ref{Subsect:LevySubordinators}), the drivers jump only at totally inaccessible times. The default times therefore avoid $\mathbb F$-predictable stopping times but not all $\mathbb F$-stopping times (Remark~\ref{rem:avoidance}(b)), and common jumps produce simultaneous defaults with an explicit probability. Since the subordinator has independent increments with respect to the reference filtration, the equivalent conditions of Lemma~\ref{lem:tractable} hold for $K^1$, $K^2$ and $K^{\{1,2\}}$, and the explicit forms \eqref{eq:biv.general.deterministic}--\eqref{eq:biv.unconditional} apply. The model of \citet{liu2020competing} (Section~\ref{subsect:Liu}) adds covariates, a static frailty factor and a time deformation. Its conditional survival factors $R^J$ are $\mathcal F_0$-measurable and, in general, random, so Lemma~\ref{lem:tractable} does not apply, yet the joint law is explicit through the general representation \eqref{eq:biv.general.R}; allowing the time deformation to jump also produces default atoms at predictable dates (Remark~\ref{rem:avoidance}(c)). For shot-noise drivers with an age-dependent kernel (Section~\ref{sn}), the predictable reductions $\Lambda^j$ are random and evolve with the path of the shocks, and the joint law is again obtained from \eqref{eq:biv.general.R}; Remark~\ref{rem:CV} describes a model from the literature in the
same situation.

\subsubsection{The Case of L\'evy Subordinators} \label{Subsect:LevySubordinators}
Let $L=(L^1,L^2)$ be a driftless bivariate L\'evy subordinator with L\'evy measure $\nu$ on
$\mathbb R_+^2\setminus\{(0,0)\}$, with $\int\bigl(1\wedge(x_1+x_2)\bigr)\,\nu(dx_1,dx_2)<\infty$.
We assume that $L$ is a L\'evy process with respect to $\mathbb F$, that is, $L$ is
$\mathbb F$-adapted and $L_u-L_s$ is independent of $\mathcal F_s$ for $0\le s\le u$. Mass of
$\nu$ on the coordinate axes produces idiosyncratic jumps, while mass in $(0,\infty)^2$
produces common jumps; the synchronized-jump specification is the case where $\nu$ is
concentrated on $(0,\infty)^2$. Set
$$
K_t^1=z_1L_t^1,\qquad K_t^2=z_2L_t^2,
\qquad z_1,z_2\ge0,
$$
so that $K^{\{1,2\}}=z_1L^1+z_2L^2$. The joint Laplace exponent is
\begin{equation}\label{eq:Joint.LE}
\psi(z_1,z_2)
=
\int_{\mathbb R_+^2\setminus\{(0,0)\}}
\left(1-e^{-z_1x_1-z_2x_2}\right)\nu(dx_1,dx_2),
\qquad
\E\bigl[e^{-z_1L_t^1-z_2L_t^2}\bigr]=e^{-t\psi(z_1,z_2)}.
\end{equation}

By independence of the increments,
$R^{\{1,2\}}(s,u)=e^{-(u-s)\psi(z_1,z_2)}$, $R^1(s,u)=e^{-(u-s)\psi(z_1,0)}$ and
$R^2(s,u)=e^{-(u-s)\psi(0,z_2)}$ are deterministic. Hence Lemma~\ref{lem:tractable}(iii)
holds for $K^1$, $K^2$ and $K^{\{1,2\}}$, with
\[
  S^{\{1,2\}}(t)=e^{-t\psi(z_1,z_2)},\qquad
  S^1(t)=e^{-t\psi(z_1,0)},\qquad
  S^2(t)=e^{-t\psi(0,z_2)},
\]
and
$\eta^{\{1,2\}}_t=\exp\{-z_1L^1_t-z_2L^2_t+t\psi(z_1,z_2)\}$. Since these survival
functions are continuous, Lemma~\ref{lem:tractable} also gives
$\Lambda^{\{1,2\}}_t=t\psi(z_1,z_2)$, $\Lambda^1_t=t\psi(z_1,0)$ and
$\Lambda^2_t=t\psi(0,z_2)$. Thus the reductions are continuous although the drivers jump:
$\tau^1$ and $\tau^2$ avoid $\mathbb F$-predictable stopping times but not all
$\mathbb F$-stopping times, as in Remark~\ref{rem:avoidance}(b).

Define
\begin{align*}
  \gamma^1&:=\psi(z_1,z_2)-\psi(0,z_2)
  =\int e^{-z_2x_2}\bigl(1-e^{-z_1x_1}\bigr)\,\nu(dx_1,dx_2),\\
  \gamma^2&:=\psi(z_1,z_2)-\psi(z_1,0)
  =\int e^{-z_1x_1}\bigl(1-e^{-z_2x_2}\bigr)\,\nu(dx_1,dx_2),\\
  \gamma^{\{1,2\}}&:=\psi(z_1,0)+\psi(0,z_2)-\psi(z_1,z_2)
  =\int \bigl(1-e^{-z_1x_1}\bigr)\bigl(1-e^{-z_2x_2}\bigr)\,\nu(dx_1,dx_2).
\end{align*}

In the notation of Remark~\ref{rem:SML}, $\Gamma^1_t=\gamma^1t$, $\Gamma^2_t=\gamma^2t$ and
$\Gamma^{\{1,2\}}_t=\gamma^{\{1,2\}}t$, and Theorem~\ref{Thm:Bivariate.Survival} gives, for
$0\le t\le\min(t_1,t_2)$,
\begin{equation}\label{eq:Levy.biv}
  \P(\tau^1>t_1,\tau^2>t_2\mid\F_t)
  =\eta^{\{1,2\}}_t\exp\Bigl\{-\gamma^1t_1-\gamma^2t_2-\gamma^{\{1,2\}}(t_1\vee t_2)\Bigr\}.
\end{equation}
Since $\gamma^1,\gamma^2,\gamma^{\{1,2\}}\ge0$, Remark~\ref{rem:SML} shows that the
unconditional law of $(\tau^1,\tau^2)$ is the bivariate Marshall--Olkin law with shock
rates $\gamma^1$, $\gamma^2$ and $\gamma^{\{1,2\}}$. This is the subordinator construction of the bivariate Marshall--Olkin law in \citet[Proposition~3 and eq.~(10)]{sun2017marshall}, applied to the subordinator $(z_1L^1,z_2L^2)$; it also appears in \citet[Theorem~2.1]{liu2020competing} (the case $\vartheta\equiv1$, $\delta_0(t)=t$ of Section~\ref{subsect:Liu}). Its $n$-dimensional version \citep[Theorem~1]{sun2017marshall} is revisited in Section~\ref{subsect:multiv.Levy}. Equation~\eqref{eq:Levy.biv} adds the $\F_t$-conditional form.

The probability of a simultaneous default is explicit. Since the jump measure of $L$ has
$\mathbb F$-compensator $ds\,\nu(dx_1,dx_2)$ and
$\E\bigl[e^{-K^{\{1,2\}}_s}\bigr]=e^{-s\psi(z_1,z_2)}$, Proposition~\ref{Prop:tau.1.equal.tau.2}
and the compensation formula give, whenever $\psi(z_1,z_2)>0$,
\[\P(\tau^1=\tau^2<\infty)
  =\gamma^{\{1,2\}}\int_0^\infty e^{-s\psi(z_1,z_2)}\,ds
  =\frac{\gamma^{\{1,2\}}}{\gamma^1+\gamma^2+\gamma^{\{1,2\}}}.
\]
This agrees with \citet[Proposition~2(iii)]{sun2017marshall} and, for constant intensities, with \citet[Example~2.10]{ProtterQuintos2024}. In particular, if $z_1,z_2>0$, simultaneous defaults occur with positive probability if
and only if $\nu\bigl((0,\infty)^2\bigr)>0$.

\paragraph{Compound Poisson drivers.}
A specification in which every jump is common is obtained by letting
\[
  L^1_t=\sum_{i=1}^{N_t}\xi_i,\qquad L^2_t=\sum_{i=1}^{N_t}\zeta_i,
\]
where $N$ is a Poisson process with intensity $\lambda>0$, the marks $(\xi_i)_{i\ge1}$ and
$(\zeta_i)_{i\ge1}$ are i.i.d.\ with laws $\rho_1$ and $\rho_2$ on $(0,\infty)$, and $N$,
$(\xi_i)$, $(\zeta_i)$ are independent. Then $(L^1,L^2)$ is a compound Poisson process,
hence a driftless bivariate L\'evy subordinator, with L\'evy measure
$\nu(dx_1,dx_2)=\lambda\,\rho_1(dx_1)\,\rho_2(dx_2)$ concentrated on $(0,\infty)^2$, and we
assume as above that it is L\'evy with respect to $\mathbb F$. Writing
$\varphi_1(z):=\E\bigl[e^{-z\xi_1}\bigr]$ and $\varphi_2(z):=\E\bigl[e^{-z\zeta_1}\bigr]$,
\eqref{eq:Joint.LE} gives $\psi(z_1,z_2)=\lambda\bigl(1-\varphi_1(z_1)\varphi_2(z_2)\bigr)$, so that
\[
  \gamma^1=\lambda\,\varphi_2(z_2)\bigl(1-\varphi_1(z_1)\bigr),\qquad
  \gamma^2=\lambda\,\varphi_1(z_1)\bigl(1-\varphi_2(z_2)\bigr),\qquad
  \gamma^{\{1,2\}}=\lambda\bigl(1-\varphi_1(z_1)\bigr)\bigl(1-\varphi_2(z_2)\bigr),
\]
and, for $z_1,z_2>0$,
\[
  \P(\tau^1=\tau^2<\infty)
  =\frac{\bigl(1-\varphi_1(z_1)\bigr)\bigl(1-\varphi_2(z_2)\bigr)}
         {1-\varphi_1(z_1)\varphi_2(z_2)}<1 .
\]
Thus, although the drivers always jump together, defaults need not be simultaneous: given
$\F_\infty$, a name that has survived up to the $i$-th jump time of $N$ defaults at that
time with probability $1-e^{-z_1\xi_i}$ for name~1 and $1-e^{-z_2\zeta_i}$ for name~2,
independently across names. A common shock may therefore produce no default, one default,
or two simultaneous defaults, as noted in Remark~\ref{rem:Chaieb}.

\subsubsection{The Construction of Liu (2020)}\label{subsect:Liu}
\citet{liu2020competing} models two durations as first-passage times of the drivers
\[
  K^i_t := \phi_i(X)\,L^i_{\vartheta\delta_0(t)},\qquad i=1,2,
\]
where $L=(L^1,L^2)$ is a bivariate L\'evy subordinator, which we take driftless, with L\'evy measure $\nu$ and joint Laplace exponent $\psi$ as in \eqref{eq:Joint.LE} (\citet{liu2020competing} also allows a drift, which only adds a linear term to $\psi$ and to $\gamma^1,\gamma^2$), $X$ denotes observed covariates, the
nonnegative functions $\phi_i$ encode the covariate effects, $\vartheta\ge0$ is an unobserved static heterogeneity (frailty) factor, and $\delta_0$ is a deterministic, nondecreasing time deformation; \citet{liu2020competing} denotes $\vartheta$ and $\delta_0$ by $\nu$ and $\Lambda_0$, respectively. The variables $X$, $\vartheta$, the process $L$ and the thresholds $(\Theta^1,\Theta^2)$ are mutually independent. \citet{liu2020competing} assumes that $\delta_0$ is continuously differentiable; here we only require it to be right-continuous, with $\delta_0(0)=0$.

To place this model in our framework, we let the static variables be known at time zero: let
$\mathbb F$ be the augmented filtration generated by $\sigma(X,\vartheta)$ and the process
$t\mapsto L_{\vartheta\delta_0(t)}$. Then $K^1,K^2$ are $\mathbb F$-adapted, c\`adl\`ag and
increasing with $K^i_0=0$, and $(\Theta^1,\Theta^2)$ is independent of $\mathcal F_\infty$.
Write $\phi_i:=\phi_i(X)$, and let $\gamma^1,\gamma^2,\gamma^{\{1,2\}}$ be the rates of
Section~\ref{Subsect:LevySubordinators} evaluated at $(z_1,z_2)=(\phi_1,\phi_2)$; they are
now $\mathcal F_0$-measurable random variables.

Since $L$ is independent of $(X,\vartheta)$ and has independent increments, freezing
$(X,\vartheta)$ and using the independence of $L_{\vartheta\delta_0(u)}-L_{\vartheta\delta_0(s)}$
from $\{L_{\vartheta\delta_0(r)}:r\le s\}$ gives, for $0\le s\le u$,
\[
    R^1(s,u)=e^{-\vartheta(\delta_0(u)-\delta_0(s))\psi(\phi_1,0)},\quad
    R^2(s,u)=e^{-\vartheta(\delta_0(u)-\delta_0(s))\psi(0,\phi_2)},
\]
\[
    R^{\{1,2\}}(s,u)=e^{-\vartheta(\delta_0(u)-\delta_0(s))\psi(\phi_1,\phi_2)}.
\]
These identities, obtained for the raw filtration, also hold for its augmentation: the
argument of \citet[Chapter~I, Theorem~31]{protter2005stochastic} uses only independent increments and a nonvanishing characteristic function, so, applied with $(X,\vartheta)$ frozen and using the right-continuity of $\delta_0$, it shows that the completed filtration is already right-continuous.

These factors are $\mathcal F_0$-measurable. Excluding the trivial case $\delta_0\equiv0$, the
factor $R^J$ is deterministic if and only if the corresponding product $\vartheta\psi(\phi_1,\phi_2)$,
$\vartheta\psi(\phi_1,0)$ or $\vartheta\psi(0,\phi_2)$ is a.s.\ constant; otherwise the
equivalent conditions of Lemma~\ref{lem:tractable} fail for $K^J$. For instance, when $\delta_0$
is continuous, conditionally on $\mathcal F_0$ the jump measure of $L_{\vartheta\delta_0(\cdot)}$
has compensator $\vartheta\,d\delta_0(s)\,\nu(dx_1,dx_2)$; since $K^1$ has no continuous part,
$\Lambda^1=A^I$ and
\[
  \Lambda^1_t=\vartheta\,\delta_0(t)\int\bigl(1-e^{-\phi_1x_1}\bigr)\,\nu(dx_1,dx_2)
  =\vartheta\,\delta_0(t)\,\psi(\phi_1,0),
\]
which is random unless $\vartheta\psi(\phi_1,0)$ is a.s.\ constant.

The general representation \eqref{eq:biv.general.R} nevertheless gives a closed form. For
$0\le t\le t_1\le t_2$, the factor $R^2(t_1,t_2)$ is a function of $(X,\vartheta)$, hence
$\mathcal F_0$-measurable and in particular $\mathcal F_t$-measurable, so it can be taken out of
the conditional expectation in \eqref{eq:biv.general.R}. Together with \eqref{eq:R.identity}
applied to $K^{\{1,2\}}$, this gives
$\P(\tau^1>t_1,\tau^2>t_2\mid\F_t)=e^{-K^{\{1,2\}}_t}R^{\{1,2\}}(t,t_1)R^2(t_1,t_2)$.
Using $\psi(\phi_1,\phi_2)=\gamma^1+\gamma^2+\gamma^{\{1,2\}}$ and
$\psi(0,\phi_2)=\gamma^2+\gamma^{\{1,2\}}$, and treating $t_2<t_1$ symmetrically, we obtain,
for $0\le t\le\min(t_1,t_2)$,
\begin{equation}\label{eq:Liu.conditional}
  \P(\tau^1>t_1,\tau^2>t_2\mid\F_t)
  =e^{-K^{\{1,2\}}_t+\vartheta\delta_0(t)\psi(\phi_1,\phi_2)}
   \exp\Bigl\{-\vartheta\bigl[\gamma^1\delta_0(t_1)+\gamma^2\delta_0(t_2)
   +\gamma^{\{1,2\}}\delta_0(t_1\vee t_2)\bigr]\Bigr\}.
\end{equation}
Taking $t=0$ and integrating out $\vartheta$, which is independent of $X$, with
$\mathcal L_\vartheta(u):=\E[e^{-u\vartheta}]$,
\[
  \P(\tau^1>t_1,\tau^2>t_2\mid X=x)
  =\mathcal L_\vartheta\Bigl(\gamma^1\delta_0(t_1)+\gamma^2\delta_0(t_2)
   +\gamma^{\{1,2\}}\delta_0(t_1\vee t_2)\Bigr),
\]
with $\gamma^1,\gamma^2,\gamma^{\{1,2\}}$ evaluated at $(\phi_1(x),\phi_2(x))$. This is the extended Marshall--Olkin law
of \citet[Theorem~2.1]{liu2020competing}. Thus Liu's model is a generalized Cox model whose conditional survival factors are $\mathcal F_0$-measurable and, in general, random:
Lemma~\ref{lem:tractable} does not apply, yet the joint law is explicit through
\eqref{eq:biv.general.R}.

\paragraph{Time deformations with jumps.}
Our framework does not require $\delta_0$ to be continuous. If $\delta_0$ jumps at $t_0>0$, then,
since $L$ has no fixed jump times and is independent of $\mathcal F_0$,
\[
  \P(\tau^1=t_0)
  =\E\Bigl[e^{-\vartheta\delta_0(t_0-)\psi(\phi_1,0)}-e^{-\vartheta\delta_0(t_0)\psi(\phi_1,0)}\Bigr],
\]
which is positive as soon as $\P\bigl(\vartheta\,\psi(\phi_1,0)>0\bigr)>0$. Then $\tau^1$ has an
atom at the deterministic, hence predictable, date $t_0$, as in Remark~\ref{rem:avoidance}(c);
similarly for $\tau^2$. This feature is excluded by the differentiability assumption of
\citet{liu2020competing}.

\subsubsection{The Case of Shot Noise Processes} \label{sn}
Let $(\theta_i)_{i\ge1}$ be $\mathbb F$-stopping times with values in $(0,\infty]$, with $\theta_i\le\theta_{i+1}$ and strict inequality on $\{\theta_i<\infty\}$, and without finite accumulation point, and
let each mark $\xi_i$ be real-valued and $\mathcal F_{\theta_i}$-measurable. The jump measure of
$(\theta_i,\xi_i)_{i\ge1}$ is
\[
  \mu\bigl((0,t]\times C\bigr):=\sum_{i\ge1}\ind_{\{\theta_i\le t\}}\ind_{\{\xi_i\in C\}},
  \qquad t\ge0,\ C\in\mathcal B(\mathbb R).
\]
Throughout this subsection, the $\mathbb F$-compensator $\nu$ of $\mu$ is assumed to be deterministic,
locally finite and diffuse in time:
\[
  \nu\bigl((0,T]\times\mathbb R\bigr)<\infty\quad\text{for all }T>0,
  \qquad
  \nu\bigl(\{t\}\times\mathbb R\bigr)=0\quad\text{for all }t\ge0 .
\]
All the explicit formulas below follow from the following exponential formula.

\begin{Lemma}\label{lem:SN.exponential.formula}
For every measurable $f:\mathbb R_+\times\mathbb R\to\mathbb R_+$ and $0\le t\le T$,
\begin{equation}\label{eq:SN.exponential.formula}
  \E\Bigl[\exp\Bigl\{-\int_{(t,T]\times\mathbb R}f(s,x)\,\mu(ds,dx)\Bigr\}\Bigm|\F_t\Bigr]
  =\exp\Bigl\{\int_{(t,T]\times\mathbb R}\bigl(e^{-f(s,x)}-1\bigr)\,\nu(ds,dx)\Bigr\}.
\end{equation}
\end{Lemma}

\begin{proof}
For $r\in[t,T]$, set
\[
  A_r:=\int_{(t,r]\times\mathbb R}f\,d\mu,\qquad
  B_r:=\int_{(t,r]\times\mathbb R}\bigl(1-e^{-f}\bigr)\,d\nu,\qquad
  Y_r:=e^{-A_r+B_r}.
\]
Since $\nu$ is deterministic and diffuse, $B$ is deterministic, continuous and bounded by
$\nu\bigl((t,T]\times\mathbb R\bigr)$, while $A$ is a pure-jump process with jumps
$f(\theta_i,\xi_i)$ at the shock times $\theta_i\in(t,T]$. Hence $dY_r=Y_{r-}\,dB_r$ between
shocks and $\Delta Y_{\theta_i}=Y_{\theta_i-}\bigl(e^{-f(\theta_i,\xi_i)}-1\bigr)$ at each shock, so
\[
  Y_r=1+\int_{(t,r]\times\mathbb R}Y_{s-}\bigl(e^{-f(s,x)}-1\bigr)\,(\mu-\nu)(ds,dx),
  \qquad t\le r\le T .
\]
The integrand is predictable and bounded by $e^{B_T}$, and
$\nu\bigl((t,T]\times\mathbb R\bigr)<\infty$. By the defining property of the compensator, the
integral is therefore a martingale on $[t,T]$, as in the proof of
\citet[Proposition~3.18]{gueye2021generalized}. Hence $\E[Y_T\mid\mathcal F_t]=Y_t=1$, and
\eqref{eq:SN.exponential.formula} follows because $B_T$ is deterministic.
\finproof
\end{proof}

Note that without diffuseness, $B$ would jump at the atoms of $\nu$, and \eqref{eq:SN.exponential.formula}
would fail.

For $j=1,2$, let $G^j:\mathbb R_+\times\mathbb R\to\mathbb R_+$ be a finite measurable kernel of the form
\begin{equation}\label{eq:SN.kernel}
  G^j(t,x)=G^j(0,x)+\int_0^t g^j(r,x)\,dr,
\end{equation}
with $g^j\ge0$ measurable, and define the shot-noise drivers
\begin{equation}\label{eq:SN.process}
  K^j_t:=\int_{(0,t]\times\mathbb R}G^j(t-s,x)\,\mu(ds,dx)
  =\sum_{i:\,\theta_i\le t}G^j(t-\theta_i,\xi_i),\qquad t\ge0 .
\end{equation}
Each $K^j$ is $\mathbb F$-adapted, c\`adl\`ag and increasing, with $K^j_0=0$. Set
\[
  k^j_t:=\int_{(0,t]\times\mathbb R}g^j(t-s,x)\,\mu(ds,dx)
  =\sum_{i:\,\theta_i\le t}g^j(t-\theta_i,\xi_i).
\]
By \eqref{eq:SN.kernel} and the substitution $v=\theta_i+r$,
$G^j(t-\theta_i,\xi_i)=G^j(0,\xi_i)+\int_{\theta_i}^t g^j(v-\theta_i,\xi_i)\,dv$. Summing over the
finitely many $\theta_i\le t$ and interchanging the finite sum and the integral gives
\begin{equation}\label{eq:SN.decomposition}
  K^j_t=\int_0^t k^j_v\,dv+\sum_{i:\,\theta_i\le t}G^j(0,\xi_i).
\end{equation}
A shock $(\theta_i,\xi_i)$ therefore acts on name $j$ in two ways: through an immediate impact, the
jump $\Delta K^j_{\theta_i}=G^j(0,\xi_i)$, and through a persistent one, the contribution
$g^j(t-\theta_i,\xi_i)$ to the shot-noise intensity $k^j_t$ at later times $t$
\citep[see][Lemma~2.2]{sss}. We say that the kernel $G^j$ has \emph{memory} when $g^j$ does not
vanish. Since both drivers are built on the same $\mu$, every shock may affect both names.

\paragraph{Conditional survival and tractability.}
Fix $0\le s\le u$. Splitting the shocks into those in $(0,s]$ and those in $(s,u]$,
\[
  K^j_u-K^j_s
  =\int_{(0,s]\times\mathbb R}\bigl(G^j(u-r,x)-G^j(s-r,x)\bigr)\,\mu(dr,dx)
   +\int_{(s,u]\times\mathbb R}G^j(u-r,x)\,\mu(dr,dx),
\]
where the first term is $\mathcal F_s$-measurable. Applying \eqref{eq:SN.exponential.formula} to the
second term, with $f(r,x)=G^j(u-r,x)$, gives
\begin{equation}\label{eq:SN.R}
\begin{split}
  R^j(s,u)
  &=\exp\Bigl\{-\int_{(0,s]\times\mathbb R}\bigl(G^j(u-r,x)-G^j(s-r,x)\bigr)\,\mu(dr,dx)\Bigr\}\\
  &\quad\times\exp\Bigl\{\int_{(s,u]\times\mathbb R}\bigl(e^{-G^j(u-r,x)}-1\bigr)\,\nu(dr,dx)\Bigr\}.
\end{split}
\end{equation}
Since $K^j_0=0$, taking $s=0$ yields
\begin{equation}\label{eq:SN.S}
  S^j(u)=\exp\Bigl\{\int_{(0,u]\times\mathbb R}\bigl(e^{-G^j(u-r,x)}-1\bigr)\,\nu(dr,dx)\Bigr\}.
\end{equation}
Equivalently,
\[
\begin{aligned}
  \P(\tau^j>u\mid\mathcal F_s)
  &=e^{-K^j_s}R^j(s,u)\\
  &=\exp\Bigl\{-\int_{(0,s]\times\mathbb R}G^j(u-r,x)\,\mu(dr,dx)
   +\int_{(s,u]\times\mathbb R}\bigl(e^{-G^j(u-r,x)}-1\bigr)\,\nu(dr,dx)\Bigr\},
\end{aligned}
\]
which is the martingale $\zeta_s(u)$ of \citet[proof of Proposition~3.18]{gueye2021generalized}; in
their notation, $c(u)=\zeta_0(u)=S^j(u)$ and $L_s(u)=\P(\tau^j>u\mid\mathcal F_s)/S^j(u)$. This
fixed-maturity martingale should not be confused with the diagonal process
$s\mapsto e^{-K^j_s}/S^j(s)$, which, by Lemma~\ref{lem:tractable}(ii), is a martingale only when the
equivalent conditions of that lemma hold.

The two factors in \eqref{eq:SN.R} have a direct interpretation. The second is deterministic and, by
\eqref{eq:SN.exponential.formula}, is the conditional Laplace transform of the hazard generated by
the shocks arriving in $(s,u]$. The first is random. By \eqref{eq:SN.kernel},
$G^j(u-r,x)-G^j(s-r,x)=\int_s^u g^j(v-r,x)\,dv$, so its exponent equals
$\int_s^u\sum_{i:\,\theta_i\le s}g^j(v-\theta_i,\xi_i)\,dv$, the hazard on $(s,u]$ generated by the
shocks that arrived before $s$, which is known at time $s$.

By Lemma~\ref{lem:tractable}(iii), the equivalent conditions of Lemma~\ref{lem:tractable} hold for
$K^j$ if and only if $R^j(s,u)$ is deterministic for all $0\le s\le u$, that is, by \eqref{eq:SN.R},
if and only if the integral in the first factor is a.s.\ constant. This integral vanishes on the event $\{\mu((0,s]\times\mathbb R)=0\}$ that no shock occurs in
$(0,s]$. Taking $f=c\,\ind_{(0,s]}$ in \eqref{eq:SN.exponential.formula} shows that
$\mu((0,s]\times\mathbb R)$ is Poisson with mean $\nu((0,s]\times\mathbb R)$, so this event has
probability $e^{-\nu((0,s]\times\mathbb R)}>0$.
Hence the integral is a.s.\ constant if and only if it vanishes a.s., that is, if and only if $\mu$ a.s.\ does
not charge the set $\{(r,x)\in(0,s]\times\mathbb R:G^j(u-r,x)>G^j(s-r,x)\}$. Since
$\E[\mu(D)]=\nu(D)$ for every measurable $D$, this happens if and only if that set is $\nu$-null.
Letting $s\le u$ range over the rationals and using the continuity of $G^j$ in its first argument,
we conclude that \emph{the equivalent conditions of Lemma~\ref{lem:tractable} hold for $K^j$ if and
only if, for $\nu$-a.e.\ $(r,x)$, $G^j(t,x)=G^j(0,x)$ for all $t\ge0$}, that is, if and only if the
kernel has no memory on the support of $\nu$.

\paragraph{Predictable reductions and avoidance.}
By \eqref{eq:SN.decomposition}, $K^{j,c}=\int_0^\cdot k^j_v\,dv$, and
$I^j=\int_{(0,\cdot]\times\mathbb R}\bigl(1-e^{-G^j(0,x)}\bigr)\,\mu(ds,dx)$ has compensator
$A^{I,j}=\int_{(0,\cdot]\times\mathbb R}\bigl(1-e^{-G^j(0,x)}\bigr)\,\nu(ds,dx)$. Hence
$\Lambda^j=K^{j,c}+A^{I,j}$ reads
\begin{equation}\label{eq:SN.reduction}
  \Lambda^j_t=\int_0^t k^j_v\,dv+\int_{(0,t]\times\mathbb R}\bigl(1-e^{-G^j(0,x)}\bigr)\,\nu(ds,dx),
\end{equation}
as in \citet[Proposition~3.17]{gueye2021generalized}. Both terms are continuous, so $\tau^j$ avoids
$\mathbb F$-predictable stopping times (Remark~\ref{rem:avoidance}(b)); it avoids all
$\mathbb F$-stopping times if and only if $G^j(0,\cdot)=0$, $\nu$-a.e. The second term is
deterministic, whereas the first, the integrated shot-noise intensity, is random as soon as the
kernel has memory. The example therefore contains three regimes.
\begin{itemize}
  \item \emph{No memory}, $G^j(t,x)=G^j(0,x)$. The conditions of Lemma~\ref{lem:tractable} hold
  for $K^1$, $K^2$ and $K^{\{1,2\}}$, whose kernel is $G^1+G^2$, and the explicit forms of
  Theorem~\ref{Thm:Bivariate.Survival} apply. If $\nu(ds,dx)=ds\,\rho(dx)$, then $(K^1,K^2)$ is
  the driftless L\'evy subordinator of Section~\ref{Subsect:LevySubordinators} with $z_1=z_2=1$
  and L\'evy measure the image of $\rho$ under $x\mapsto\bigl(G^1(0,x),G^2(0,x)\bigr)$, restricted
  to $\mathbb R^2_+\setminus\{(0,0)\}$.
  \item \emph{No immediate impact}, $G^j(0,\cdot)=0$. Then $K^j$ is continuous and $\tau^j$ is a
  classical Cox time with the random shot-noise intensity $k^j$, and $\Lambda^j=K^j$
  (Remark~\ref{rem:avoidance}(a)).
  \item \emph{General case.} The reductions $\Lambda^j$ are random, Lemma~\ref{lem:tractable}
  fails, and default occurs at shock times with positive probability.
\end{itemize}

\paragraph{Joint law.}
For $t_1,t_2\ge0$ put $T:=t_1\vee t_2$ and
\begin{equation}\label{eq:SN.H}
  H_{t_1,t_2}(s,x):=\ind_{\{s\le t_1\}}G^1(t_1-s,x)+\ind_{\{s\le t_2\}}G^2(t_2-s,x),
\end{equation}
so that $K^1_{t_1}+K^2_{t_2}=\int_{(0,T]\times\mathbb R}H_{t_1,t_2}(s,x)\,\mu(ds,dx)$. Since
$\tau^1$ and $\tau^2$ are conditionally independent given $\mathcal F_\infty$,
\eqref{eq:SN.exponential.formula} gives, for $0\le t\le t_1\wedge t_2$,
\begin{equation}\label{eq:SN.joint}
\begin{split}
  \P(\tau^1>t_1,\tau^2>t_2\mid\F_t)
  &=\exp\Bigl\{-\int_{(0,t]\times\mathbb R}H_{t_1,t_2}(s,x)\,\mu(ds,dx)\\
  &\qquad\qquad+\int_{(t,T]\times\mathbb R}\bigl(e^{-H_{t_1,t_2}(s,x)}-1\bigr)\,\nu(ds,dx)\Bigr\}.
\end{split}
\end{equation}
This is the general identity \eqref{eq:biv.general.R} of Theorem~\ref{Thm:Bivariate.Survival},
made explicit through \eqref{eq:SN.R}. The explicit product forms of that theorem do not apply
here, since the reductions $\Lambda^j$ are random and Lemma~\ref{lem:tractable} fails;
nevertheless, the joint law is available in closed form. Taking $t=0$ and comparing with \eqref{eq:SN.S},
we obtain, for all $t_1,t_2\ge0$,
\begin{equation}\label{eq:SN.dependence}
\begin{split}
  \P(\tau^1>t_1,\tau^2>t_2)
  &=S^1(t_1)\,S^2(t_2)\\
  &\quad\times\exp\Bigl\{\int_{(0,t_1\wedge t_2]\times\mathbb R}
   \bigl(1-e^{-G^1(t_1-s,x)}\bigr)\bigl(1-e^{-G^2(t_2-s,x)}\bigr)\,\nu(ds,dx)\Bigr\}.
\end{split}
\end{equation}
Hence $(\tau^1,\tau^2)$ is positively quadrant dependent, and only shocks that affect both names
contribute to the dependence factor. Without memory, this factor depends on $(t_1,t_2)$ only through
$t_1\wedge t_2$, as in the Marshall--Olkin form of Remark~\ref{rem:SML}. With memory it depends on both
maturities, because a common shock at $s\le t_1\wedge t_2$ keeps raising the hazard of each name up to
its own horizon. In general, the joint law is then not of the form of Remark~\ref{rem:SML}.

\paragraph{Simultaneous defaults.}
The drivers jump together at each $\theta_i$, with $\Delta K^j_{\theta_i}=G^j(0,\xi_i)$. The
integrand $e^{-K^{\{1,2\}}_{s-}}\bigl(1-e^{-G^1(0,x)}\bigr)\bigl(1-e^{-G^2(0,x)}\bigr)$ is
predictable, and $S^{\{1,2\}}$ is continuous because $\nu$ is diffuse. Hence
Proposition~\ref{Prop:tau.1.equal.tau.2} and the compensation formula give
\begin{equation}\label{eq:SN.simultaneous}
  \P(\tau^1=\tau^2<\infty)
  =\int_{(0,\infty)\times\mathbb R}S^{\{1,2\}}(s)\bigl(1-e^{-G^1(0,x)}\bigr)\bigl(1-e^{-G^2(0,x)}\bigr)\,
   \nu(ds,dx),
\end{equation}
where, by \eqref{eq:SN.S} applied to the kernel $G^1+G^2$,
$S^{\{1,2\}}(s)=\exp\bigl\{\int_{(0,s]\times\mathbb R}\bigl(e^{-G^1(s-r,x)-G^2(s-r,x)}-1\bigr)\,\nu(dr,dx)\bigr\}$.
Thus $\P(\tau^1=\tau^2<\infty)>0$ if and only if
$\nu\bigl((0,\infty)\times\{x:G^1(0,x)\,G^2(0,x)>0\}\bigr)>0$. Simultaneous defaults are produced
only by the immediate impacts, while the persistent impacts create dependence through
\eqref{eq:SN.dependence} without simultaneous defaults. In the memoryless case with
$\nu(ds,dx)=ds\,\rho(dx)$, \eqref{eq:SN.simultaneous} reduces to the formula
$\gamma^{\{1,2\}}/(\gamma^1+\gamma^2+\gamma^{\{1,2\}})$ of
Section~\ref{Subsect:LevySubordinators}.

\paragraph{Nonhomogeneous Poisson arrivals.}
A standard specification takes the shock times $\theta_i$ to be the jump times of a nonhomogeneous
Poisson process with a nonnegative locally integrable intensity function $\lambda^N$ with $\int_0^\infty\lambda^N(s)\,ds>0$, the marks $\xi_i$ i.i.d.\ with law
$\rho$ on $(0,\infty)$ and independent of the arrivals, and multiplicative kernels
$G^j(t,x)=x\,b^j(t)$, where $b^j\ge0$ is nondecreasing and absolutely continuous. Then
$\nu(ds,dx)=\lambda^N(s)\,ds\,\rho(dx)$ is deterministic and diffuse. Writing
$\mathcal L_\xi(z):=\E\bigl[e^{-z\xi_1}\bigr]$, \eqref{eq:SN.S} and \eqref{eq:SN.joint} with $t=0$
become
\[
  S^j(u)=\exp\Bigl\{-\int_0^u\lambda^N(s)\bigl(1-\mathcal L_\xi(b^j(u-s))\bigr)\,ds\Bigr\},
\]
\[
  \P(\tau^1>t_1,\tau^2>t_2)
  =\exp\Bigl\{-\int_0^{t_1\vee t_2}\lambda^N(s)\Bigl(1-\mathcal L_\xi\bigl(
   \ind_{\{s\le t_1\}}b^1(t_1-s)+\ind_{\{s\le t_2\}}b^2(t_2-s)\bigr)\Bigr)\,ds\Bigr\}.
\]
The kernel of name $j$ has memory as soon as $b^j$ is not constant, and simultaneous defaults occur
with positive probability if and only if $b^1(0)\,b^2(0)>0$. When $b^1(0)=b^2(0)=0$, shocks are
non-fatal: the default times avoid all $\mathbb F$-stopping times, and the dependence between the
two lifetimes comes only from the common cumulative effect of past shocks. Shot-noise intensities
of this type are used by \citet{lee2018dynamic} to model the impact of non-fatal common shocks on
the mortality of coupled lives.

\brem\label{rem:SN.SSS}
In \citet[eq.~(3)]{sss}, written in our notation, $\tau^j=\inf\{t\ge0:U_{\chi_j(t)}\ge\Theta^j\}$, where
$U_t=\ell t+\sum_{T_n\le t}V_n\,h(t-T_n)$ is a nondecreasing shot-noise process built from a
Poisson process with jump times $T_n$, i.i.d.\ positive marks $V_n$, a nonnegative nondecreasing
response function $h$ and a drift $\ell\ge0$. The deterministic, strictly increasing time changes
$\chi_j$ are chosen to match prescribed marginal laws \citep[Lemma~3.4]{sss}, and the thresholds
$\Theta^j$ are independent unit exponentials, independent of $U$. This is a generalized Cox model
with $K^j=U_{\chi_j(\cdot)}$, so Theorem~\ref{Thm:Bivariate.Survival} and
Section~\ref{sect:Generalization} apply to it. When $\chi_1=\chi_2$ is the identity and $\ell=0$, it is the case of the nonhomogeneous Poisson specification above with constant $\lambda^N$ and $b^1=b^2=h$, provided $h$ is absolutely continuous. In general, the two constructions differ in complementary ways. A jump
of $U$ at $T_n$ can trigger name $j$ only at the calendar date $\chi_j^{-1}(T_n)$. For continuous
$\chi_j$, $\P(\tau^1=\tau^2<\infty)>0$ therefore requires $\chi_1^{-1}$ and $\chi_2^{-1}$ to
agree on a set of positive Lebesgue measure. Accordingly, the joint defaults of
\citet[Section~3.4]{sss} are coincidences of the transformed times $p_j(\tau^j)$, where $p_j$ is
the distribution function of $\tau^j$. In \eqref{eq:SN.process}, by contrast, a common shock arrives
at the same date for all names, so names with different marginal laws and different responses
$G^1\neq G^2$ can still default at the same calendar time.
\erem

\begin{rem}[Relation to \citet{coculescu2024default}]\rm\label{rem:CV}
    Outside the shot-noise setting, another model from the literature in which the reductions
    $\Lambda^i$ are random is the reference model of \citet{coculescu2024default}, whose reference
    measure corresponds to our $\P$. In our notation, their drivers are
    \[
      K^i_t=\int_0^t\alpha^i_s\,ds+h^i_{T_i}\ind_{\{T_i\le t\}},
    \]
    where $\alpha^i$ and $h^i$ are bounded nonnegative $\mathbb F$-predictable
    processes, and $T_i$ is a totally inaccessible $\mathbb F$-stopping time with
    bounded intensity $\kappa^i$; distinct shock times never coincide, i.e.\
    $\P(T_i=T_j<\infty)=0$ for $i\neq j$. Here $K^{i,c}=\int_0^\cdot\alpha^i_s\,ds$, and
    $I^i=\bigl(1-e^{-h^i_{T_i}}\bigr)\ind_{\{T_i\le\cdot\}}$ has compensator
    $A^{I,i}=\int_0^{\cdot\wedge T_i}\bigl(1-e^{-h^i_s}\bigr)\kappa^i_s\,ds$, since $h^i$
    is predictable. Hence $\Lambda^i=K^{i,c}+A^{I,i}$ reads
    \[
      \Lambda^i_t=\int_0^t\lambda^i_s\,ds,\qquad
      \lambda^i_t=\alpha^i_t+\bigl(1-e^{-h^i_t}\bigr)\ind_{\{T_i\ge t\}}\kappa^i_t ,
    \]
    in agreement with \citet[Lemma~3.1]{coculescu2024default}. Thus $\Lambda^i$ is
    continuous, as in Remark~\ref{rem:avoidance}(b), but random: since $\lambda^i$ drops
    at the random time $T_i$ whenever $h^i\kappa^i>0$, $\Lambda^i$ is random even when
    $\alpha^i$, $h^i$ and $\kappa^i$ are deterministic. Hence Lemma~\ref{lem:tractable}
    fails, and the general representation~\eqref{eq:biv.general.R} applies. The
    conditional survival factors are available under their default-adjusted measure
    $\overline{\P}^{\,i}$ \citep[Definition~3.3, with $C=\{i\}$]{coculescu2024default}:
    $R^i(s,u)=\overline{\E}^{\,i}\bigl[e^{-(\Lambda^i_u-\Lambda^i_s)}\mid\mathcal F_s\bigr]$,
    the conditional version of their formula~(3.11). As discussed in the introduction,
    their construction then adds contagion through a change of measure, while
    excluding the common jumps that produce simultaneous defaults here.
\erem

\section{Generalization of the Construction for \texorpdfstring{$n$}{n} Components}\label{sect:Generalization}
We now generalize our construction to the case of $n$ components.
To maintain consistency and clarity, we explicitly extend the definitions, hypotheses, and results introduced in the bivariate case (Section \ref{SecBiv}). Theorem~\ref{Thm:Multivariate.Recursion} gives the $\mathcal F_t$-conditional joint survival law for general drivers. Under the conditions of Lemma~\ref{lem:tractable} for the aggregates, it reduces to the ordered product of Theorem~\ref{Thm:Multivariate.Survival}, which Corollary~\ref{Cor:Multivariate.Joint.Distr} rewrites in M\"obius coordinates. Proposition~\ref{prop:continuous.reductions.MO} gives a condition on the drivers under which these coordinates define an independent-shock law, and Example~\ref{ex:atomic.signed.mobius} shows that default atoms at predictable dates can take the model outside that class. Remark~\ref{rem:independent.blocks} isolates a tractable subclass with stochastic continuous parts.

We consider default times $\tau^1, \ldots, \tau^n$, where each $\tau^i$ is defined as the first hitting time
\[
\tau^i := \inf \{ t \geq 0 : K_t^i \geq \Theta^i \},
\]
with $K^i$ an $\mathbb{F}$-adapted, càdlàg, increasing process with $K^i_0=0$, and $\Theta^1,\dots,\Theta^n$ independent unit exponential random variables such that the vector $(\Theta^1,\dots,\Theta^n)$ is independent of $\mathcal F_\infty$. The processes $K^1, \dots, K^n$ are allowed to be dependent, which may induce dependence among the default times $\tau^1, \dots, \tau^n$.

\par\medskip

For every nonempty $J\subseteq V:=\{1,\dots,n\}$, define
$$
K_t^J:=\sum_{j\in J}K_t^j,
\qquad
\tau^J:=\min_{j\in J}\tau^j.
$$
Since $\P(\tau^J>t\mid\mathcal F_\infty)=e^{-K^J_t}$, Lemmas~\ref{lem:Lambda.Doleans.Behav} and~\ref{lem:tractable} apply to each $\tau^J$ with $K$ replaced by $K^J$. Accordingly, we write $S^J$, $R^J$, $\Lambda^J$ and $\eta^J$ for the objects associated with $\tau^J$ and $K^J$, with the conventions $K^\emptyset\equiv0$ and $S^\emptyset\equiv1$. In particular, without (H1)--(H2) or any avoidance condition,
\begin{equation}\label{eq:multiv.aggregate}
Z_t^J:=\P(\tau^J>t\mid\F_t)=e^{-K_t^J}=\eta^J_t\,\mathcal E\left(-\Lambda^J\right)_t .
\end{equation}

Fix maturities $t_1,\dots,t_n\ge0$, let $\sigma$ be a permutation of $V$ with $t_{\sigma(1)}\le\cdots\le t_{\sigma(n)}$ (any choice in case of ties), and set
$$
s_k:=t_{\sigma(k)},\qquad
A_k:=\{\sigma(k),\ldots,\sigma(n)\},\qquad k=1,\ldots,n,\qquad A_{n+1}:=\emptyset .
$$
Thus $A_1=V$ and $A_k=\{\sigma(k)\}\cup A_{k+1}$.
The identity
$$
\sum_{i=1}^n K_{t_i}^i
=
K_{s_1}^{A_1}
+
\sum_{k=2}^n\left(K_{s_k}^{A_k}-K_{s_{k-1}}^{A_k}\right)
$$
leads to a backward recursion that does not require deterministic predictable factors.

\begin{Theorem}\label{Thm:Multivariate.Recursion}
Set $\Psi_n:=1$ and, for $k=n,\ldots,2$, define
\begin{equation}\label{eq:backward.recursion}
\Psi_{k-1}
:=
\E\left[
\exp\left\{-\left(K_{s_k}^{A_k}-K_{s_{k-1}}^{A_k}\right)\right\}
\Psi_k
\,\Bigm|\,
\F_{s_{k-1}}
\right].
\end{equation}
Then, for $0\le t\le s_1$,
\begin{equation}\label{eq:multiv.general.recursion}
\P(\tau^1>t_1,\ldots,\tau^n>t_n\mid\F_t)
=
\E\left[e^{-K_{s_1}^{A_1}}\Psi_1\mid\F_t\right].
\end{equation}
No determinism or avoidance assumption is required.
\end{Theorem}

\proof
Since $\Theta^1,\dots,\Theta^n$ are independent and $(\Theta^1,\dots,\Theta^n)$ is independent of $\F_\infty$, we have $\P(\tau^1>t_1,\ldots,\tau^n>t_n\mid\F_\infty)=e^{-\sum_{i=1}^n K^i_{t_i}}$, and hence, by the tower property,
$$
\P(\tau^1>t_1,\ldots,\tau^n>t_n\mid\F_t)
=
\E\left[e^{-\sum_{i=1}^n K_{t_i}^i}\mid\F_t\right].
$$
Insert the preceding telescoping identity and condition successively at
$\F_{s_{n-1}},\ldots,\F_{s_1}$. This gives \eqref{eq:backward.recursion} and
\eqref{eq:multiv.general.recursion}.
\finproof

Theorem~\ref{Thm:Multivariate.Survival}, Corollary~\ref{Cor:Multivariate.Joint.Distr} and Proposition~\ref{prop:continuous.reductions.MO} use the assumption that, for every nonempty $J\subseteq V$, the equivalent conditions of Lemma~\ref{lem:tractable} hold for $K^J$; each statement says which part of it is needed. By the last assertion of that lemma, (H1)--(H2) imposed on every aggregate $K^J$ provide one sufficient criterion. Section~\ref{subsect:multiv.Liu} and Remark~\ref{rem:independent.blocks} show what remains explicit without it.

\begin{rem}\label{rem:joint.jump.measure.criterion}
    The assumption of Theorem~\ref{Thm:Multivariate.Survival} concerns all $2^n-1$ aggregates and does not follow from the corresponding assumption on the marginals; see the example preceding Remark~\ref{rem:Chaieb} for $n=2$. A simple sufficient condition covering all aggregates at once is that the vector $K=(K^1,\ldots,K^n)$ has independent increments with respect to $\mathbb F$, that is, $K_u-K_s$ is independent of $\mathcal F_s$ for $0\le s\le u$. Then, for every nonempty $J\subseteq V$, $R^J(s,u)=\E\bigl[e^{-(K^J_u-K^J_s)}\bigr]$ is deterministic, so Lemma~\ref{lem:tractable}(iii) holds for $K^J$. The examples of Section~\ref{subsect:multiv.Levy} are of this type; in Section~\ref{subsect:multiv.Liu} with a random frailty, $K$ has independent increments only conditionally on $\mathcal F_0$, and $R^J$ is $\mathcal F_0$-measurable rather than deterministic.
\end{rem}

\begin{Theorem}\label{Thm:Multivariate.Survival}
Assume that, for every nonempty $J\subseteq V$, the equivalent conditions of
Lemma~\ref{lem:tractable} hold for $K^J$. Let $\sigma$, $s_k$ and $A_k$ be as above.
Then, for $0\le t\le\min_i t_i$,
\begin{equation}\label{FI}
  \P(\tau^1>t_1,\dots,\tau^n>t_n\mid\mathcal F_t)
  =\eta^V_t\prod_{k=1}^n
   \frac{S^{A_k}(s_k)}{S^{A_{k+1}}(s_k)},
  \qquad
  \eta^V_t=\frac{e^{-K^V_t}}{S^V(t)} .
\end{equation}

The unconditional law requires less: if the equivalent conditions of Lemma~\ref{lem:tractable} hold for $K^J$ only for every nonempty $J\subsetneq V$, then, without any condition on $K^V$,
\[
  \P(\tau^1>t_1,\dots,\tau^n>t_n)=\prod_{k=1}^n\frac{S^{A_k}(s_k)}{S^{A_{k+1}}(s_k)} .
\]

Conversely, if \eqref{FI} holds for all $t_1,\dots,t_n\ge0$ and $0\le t\le\min_i t_i$, then the equivalent conditions of Lemma~\ref{lem:tractable} hold for $K^J$ for every nonempty $J\subseteq V$.

\end{Theorem}

\begin{proof}
By Lemma~\ref{lem:tractable}, for every nonempty
$J$ and $s\le u$, the conditional survival factor
$R^J(s,u)=\E\bigl[e^{-(K^J_u-K^J_s)}\mid\mathcal F_s\bigr]$ is deterministic
and equals $S^J(u)/S^J(s)$. Hence, in the backward recursion of
Theorem~\ref{Thm:Multivariate.Recursion}, each $\Psi_k$ is deterministic, and by induction
\[
  \Psi_1=\prod_{k=2}^n\frac{S^{A_k}(s_k)}{S^{A_k}(s_{k-1})}.
\]
Since $\Psi_1$ is deterministic and $A_1=V$, equation \eqref{eq:multiv.general.recursion}, together with \eqref{eq:R.identity} applied to $K^V$, gives
\[
  \P(\tau^1>t_1,\dots,\tau^n>t_n\mid\mathcal F_t)
  =\E\bigl[e^{-K^V_{s_1}}\mid\mathcal F_t\bigr]\,\Psi_1
  =e^{-K^V_t}\,\frac{S^V(s_1)}{S^V(t)}\,\Psi_1 .
\]
Rearranging the factors, using $A_1=V$ and $S^{A_{n+1}}=S^\emptyset\equiv1$,
yields \eqref{FI}.

The computation of $\Psi_1$ uses only $R^{A_k}$ with $k\ge2$, and these $A_k$ are nonempty proper subsets of $V$. Taking expectations in \eqref{eq:multiv.general.recursion} then gives $\E[e^{-K^V_{s_1}}]\,\Psi_1=S^V(s_1)\Psi_1$, which is the unconditional formula and uses no condition on $K^V$.

For the converse, fix a nonempty $J\subseteq V$ and $0\le s\le u$, and take $t=s$, $t_i=s$ for $i\notin J$ and $t_i=u$ for $i\in J$, ordering the names outside $J$ first. By \eqref{eq:R.identity}, the left side of \eqref{FI} is $e^{-K^V_s}R^J(s,u)$. The ordered product telescopes, so the right side is $e^{-K^V_s}S^J(u)/S^J(s)$. Hence $R^J(s,u)=S^J(u)/S^J(s)$ is deterministic, which is Lemma~\ref{lem:tractable}(iii) for $K^J$.
\finproof
\end{proof}

\brem
    When the survival functions $S^J$ are continuous, Lemma~\ref{lem:tractable}
    gives $S^J=e^{-\Lambda^J}$, and equation \eqref{FI} reads
    \[
      \P(\tau^1>t_1,\dots,\tau^n>t_n\mid\mathcal F_t)
      =\eta^V_t\exp\Bigl\{-\sum_{k=1}^n\bigl(\Lambda^{A_k}_{s_k}-\Lambda^{A_{k+1}}_{s_k}\bigr)\Bigr\},
      \qquad \Lambda^\emptyset\equiv0 .
    \]
    In general, the jumps of $S^J$ (default atoms at fixed dates) make
    $-\log S^J$ and $\Lambda^J$ differ, and \eqref{FI} is the correct form.
\erem

The ordered product in equation \eqref{FI} is structurally related to the ordered-transform characterization of exogenous shock survival functions in \citet{scherer2019exogenous}. In the L\'evy case, the same ordered computation appears in the proof of \citet[Theorem~1]{sun2017marshall}. The objects and level of conditioning are different: \eqref{FI} is an $\F_t$-conditional identity obtained from the multiplicative decomposition of generalized Cox times and retains the dynamic factor $\eta_t^V$. In addition, the characterization of \citet{scherer2019exogenous} assumes continuous shock survival functions, whereas here the survival functions $S^J$ may jump, which corresponds to default atoms at fixed dates (see Section~\ref{subsect:multiv.Liu}). Thus the exogenous-shock characterization is a benchmark for the setting of Theorem~\ref{Thm:Multivariate.Survival}, while Theorem~\ref{Thm:Multivariate.Recursion} covers general adapted increasing drivers at the level of conditional survival recursions.

For every nonempty $B\subseteq V$, define the M\"obius coordinate, which extends the coordinates of Remark~\ref{rem:SML} to $n$ names,
\begin{equation}\label{eq:Mobius.Gamma}
    \Gamma^B_t:=\sum_{M\subseteq B}(-1)^{|B|-|M|}\log S^{V\setminus M}(t).
\end{equation}

M\"obius inversion (see \citet{Rota1964}) gives
\begin{equation}\label{eq:Mobius.inverse}
    -\log S^J(t)=\sum_{\substack{\emptyset\neq B\subseteq V\\ B\cap J\neq\emptyset}}\Gamma^B_t.
\end{equation}

\begin{Cor}\label{Cor:Multivariate.Joint.Distr}
Under the assumption of Theorem~\ref{Thm:Multivariate.Survival}, for $0\le t\le\min_i t_i$,
\begin{equation}\label{JD_mult}
\P(\tau^1>t_1,\ldots,\tau^n>t_n\mid\F_t)
=
\eta_t^V
\exp\left\{
-\sum_{\emptyset\neq B\subseteq V}\Gamma_{\max_{i\in B}t_i}^B
\right\}.
\end{equation}
\end{Cor}

\proof
 Let
$$
H^J(t):=-\log S^J(t)=\sum_{\substack{\emptyset\neq B\subseteq V\\B\cap J\neq\emptyset}}\Gamma^B_t,\qquad\text{so that } H^\emptyset\equiv0.
$$
Taking minus logarithms of the product in \eqref{FI} gives
$$
-\log\!\left(\prod_{k=1}^n
\frac{S^{A_k}(s_k)}{S^{A_{k+1}}(s_k)}\right)
=
\sum_{k=1}^n\bigl(H^{A_k}(s_k)-H^{A_{k+1}}(s_k)\bigr),
 \qquad \text{with }s_k=t_{\sigma(k)}.
$$
Fix a nonempty $B\subseteq V$. Since $A_{k+1}\subseteq A_k$, every $B$ that meets $A_{k+1}$ also meets $A_k$, so in $H^{A_k}-H^{A_{k+1}}$ these terms cancel and only the $B$ meeting $A_k$ but not $A_{k+1}$ remain. Since $A_k=\{\sigma(k)\}\cup A_{k+1}$, the term $\Gamma^B$ occurs in $H^{A_k}-H^{A_{k+1}}$ if and only if
$
 \sigma(k)\in B
 \quad\text{and}\quad
 B\cap A_{k+1}=\emptyset,
$ that is, if and only if $\sigma(k)$ is the last element of $B$ in the list $\sigma(1),\ldots,\sigma(n)$.
There is exactly one such index,
$$
 k_B:=\max\{\ell:\sigma(\ell)\in B\}.
$$
For this index, since $t_{\sigma(1)}\le\cdots\le t_{\sigma(n)}$ and $\sigma(k_B)$ is the element of $B$ with the largest position,
\[
 s_{k_B}=t_{\sigma(k_B)}=\max_{i\in B}t_i.
\]
Therefore, since each nonempty $B$ contributes to exactly one difference, namely the one with $k=k_B$,
\[
 \sum_{k=1}^n
 \left[
 H^{A_k}(s_k)-H^{A_{k+1}}(s_k)
 \right]
 =\sum_{k=1}^n\ \sum_{\substack{\emptyset\neq B\subseteq V\\ k_B=k}}\Gamma^B_{s_k}=
 \sum_{\emptyset\neq B\subseteq V}
 \Gamma^B_{\max_{i\in B}t_i}.
\]
Substitution into \eqref{FI} yields \eqref{JD_mult}.
\finproof

\begin{rem}\label{rem:MO.interpretation}
The coordinates $\Gamma^B$ in \eqref{eq:Mobius.Gamma} are defined algebraically, and M\"obius inversion alone does not make them nondecreasing. They are right-continuous with $\Gamma^B_0=0$, since each $S^J$ is right-continuous with $S^J(0)=1$. If every $\Gamma^B$ is nondecreasing, let $(E_B)_{\emptyset\neq B\subseteq V}$ be independent random times with values in $(0,\infty]$ and $\P(E_B>t)=e^{-\Gamma^B_t}$, and set
\[
  \widehat\tau^{\,i}:=\min\{E_B:\ B\subseteq V,\ i\in B\},\qquad i\in V,
\]
so that each name defaults at the first shock that affects it. Corollary~\ref{Cor:Multivariate.Joint.Distr} with $t=0$ then gives
\[
  \P(\tau^1>t_1,\ldots,\tau^n>t_n)
  =\exp\Bigl\{-\sum_{\emptyset\neq B\subseteq V}\Gamma^B_{\max_{i\in B}t_i}\Bigr\}
  =\P(\widehat\tau^{\,1}>t_1,\ldots,\widehat\tau^{\,n}>t_n),
\]
that is, $(\tau^1,\ldots,\tau^n)$ has the law of a fatal shock model with independent shocks. When, moreover, the $\Gamma^B$ are continuous and each $\tau^i$ is a.s.\ finite, this is the generalized Marshall--Olkin distribution of \citet[eq.~(5)]{scherer2019exogenous}; when the $\Gamma^B$ are absolutely continuous, it is the time-inhomogeneous Marshall--Olkin distribution of \citet[Appendix~A]{sun2017online}. In the homogeneous case $\Gamma^B_t=\gamma^B t$, it is the classical Marshall--Olkin law with shock rates $\gamma^B$ \citep[eq.~(11)]{sun2017marshall}, and \eqref{eq:Mobius.Gamma}--\eqref{eq:Mobius.inverse} reduce to the relations between its parameters in \citet[eq.~(12) and Lemma~1]{sun2017marshall}.
\end{rem}

\bp\label{prop:continuous.reductions.MO}
Assume that, for every nonempty $J\subseteq V$, the equivalent conditions of Lemma~\ref{lem:tractable} hold for $K^J$, and that $S^1,\ldots,S^n$ are continuous (equivalently, by Remark~\ref{rem:avoidance}(b) and Lemma~\ref{lem:tractable}, that $\tau^1,\ldots,\tau^n$ avoid $\ff$-predictable stopping times). For nonempty $B\subseteq V$ and $x\in\R_+^n$, set
\be\label{eq:structural.QJ}
  q_B(x):=e^{-\sum_{r\notin B}x_r}\prod_{i\in B}\bigl(1-e^{-x_i}\bigr),
  \qquad
  Q^B_t:=\sum_{0<s\le t}q_B(\Delta K_s),
\ee
where $\Delta K_s:=(\Delta K^1_s,\ldots,\Delta K^n_s)$, and let $(Q^B)^{p,\ff}$ be the $\ff$-dual predictable projection of $Q^B$ (Section~\ref{sect:Known.Results.Definitions}), which exists by the proof below. Then
\be\label{eq:structural.gamma.decomposition}
  \Gamma^B=
  \begin{cases}
    K^{j,c}+(Q^{\{j\}})^{p,\ff}, & B=\{j\},\\
    (Q^B)^{p,\ff}, & |B|\ge2.
  \end{cases}
\ee
In particular, every $\Gamma^B$ is continuous and nondecreasing, and the unconditional law of $(\tau^1,\ldots,\tau^n)$ is the fatal shock law of Remark~\ref{rem:MO.interpretation}.
\ep

\proof
Fix a nonempty $J\subseteq V$. Continuity of the marginals extends to $S^J$: since $\P(\tau^J=t)\le\sum_{i\in J}\P(\tau^i=t)=0$ for all $t$, $S^J$ is continuous. By Lemma~\ref{lem:tractable}, $\Lambda^J$ is deterministic with $\Delta\Lambda^J_t=\P(\tau^J=t)/\P(\tau^J\ge t)=0$, hence continuous, and $S^J=\mathcal E(-\Lambda^J)=e^{-\Lambda^J}$.

Applying Section~\ref{sect:Bivariate.Case} to $K^J=\sum_{i\in J}K^i$, we have $\Lambda^J=K^{J,c}+A^{I,J}$, where $A^{I,J}$ is the predictable finite-variation part of
\[
  I^J_t:=\sum_{0<s\le t}\bigl(1-e^{-\Delta K^J_s}\bigr)=\sum_{0<s\le t}f_J(\Delta K_s),
  \qquad f_J(x):=1-e^{-\sum_{i\in J}x_i}.
\]
Since $I^J$ is increasing with jumps bounded by $1$, $A^{I,J}$ is its dual predictable projection. We set $f_\emptyset:=0$ and $\Lambda^\emptyset=A^{I,\emptyset}:=0$. Since $\log S^J=-\Lambda^J$, \eqref{eq:Mobius.Gamma} gives
\[
  \Gamma^B=-\sum_{M\subseteq B}(-1)^{|B|-|M|}\Lambda^{V\setminus M}.
\]
The coefficient of $K^{r,c}$ in this sum is $-\sum_{M\subseteq B,\,r\notin M}(-1)^{|B|-|M|}=\ind_{\{B=\{r\}\}}$. Since $\sum_{M\subseteq B}(-1)^{|B|-|M|}=0$, the inclusion--exclusion formula gives $-\sum_{M\subseteq B}(-1)^{|B|-|M|}f_{V\setminus M}=q_B$. Moreover $0\le q_B\le f_V$, so $0\le Q^B\le I^V$ and $\E[Q^B_t]\le\E[A^{I,V}_t]\le\Lambda^V_t<\infty$. Hence $Q^B$ admits a dual predictable projection, and by linearity the jump contribution to $\Gamma^B$ is $(Q^B)^{p,\ff}$, which proves \eqref{eq:structural.gamma.decomposition}. Both terms are increasing, and $\Gamma^B$ is continuous as a linear combination of the continuous functions $\Lambda^J$. The last assertion is Remark~\ref{rem:MO.interpretation}.
\finproof

The decomposition \eqref{eq:structural.gamma.decomposition} has a direct interpretation. Given $\mathcal F_\infty$ and that all names are alive just before a jump time $s$, name $i$ defaults at $s$ with probability $1-e^{-\Delta K^i_s}$, independently across names, so $q_B(\Delta K_s)$ is the probability that exactly the names in $B$ default at $s$. Consequently, $\Gamma^B$ is the cumulative hazard of the event that the first default affects exactly the names in $B$:
\[
  \P\bigl(\tau^V\le t,\ \{i:\tau^i=\tau^V\}=B\bigr)=\int_{(0,t]}S^V(s)\,d\Gamma^B_s,\qquad t\ge0 .
\]
Indeed, given $\mathcal F_\infty$, the left-hand side equals $\int_{(0,t]}Z^V_{s-}\,dQ^B_s$, plus $\int_{(0,t]}Z^V_s\,dK^{j,c}_s$ when $B=\{j\}$; since $Z^V_-$ is predictable, $Q^B$ may be replaced by $(Q^B)^{p,\ff}$, and taking expectations gives the formula by \eqref{eq:structural.gamma.decomposition}. In the representation of Remark~\ref{rem:MO.interpretation}, $\Gamma^B$ is also the cumulative hazard of the shock $E_B$.

Proposition~\ref{prop:continuous.reductions.MO} requires neither independent increments nor absolutely continuous cumulative hazards, and identifies $\Gamma^B$ from the drivers. It thus extends the subordinator constructions of \citet{sun2017marshall}: for L\'evy drivers it gives $\Gamma^B_t=\gamma^Bt$ with $\gamma^B$ as in \eqref{eq:multiv.Levy.gamma}, which is their Lemma~2, and for additive subordinators it recovers the shock representation of their Section~6 and \citet[Appendix~A]{sun2017online}.

\bex\label{ex:atomic.signed.mobius}
Let $Y$ be Bernoulli with parameter $1/2$, revealed to the reference filtration at the deterministic date $t_0=1$, that is, $\mathcal F_t$ is trivial for $t<1$ and $\mathcal F_t=\sigma(Y)$ for $t\ge1$, and take
$$
 K_t^1=Y\ind_{\{t\ge1\}},\qquad
 K_t^2=(1-Y)\ind_{\{t\ge1\}}.
$$
For all $t\ge1$,
$$
 S^1(t)=S^2(t)=\frac{1+e^{-1}}{2},\qquad
 S^{\{1,2\}}(t)=e^{-1}.
$$
Every $R^J(s,u)$ is deterministic, so the conditions of Lemma~\ref{lem:tractable} hold for $K^1$, $K^2$ and $K^{\{1,2\}}=\ind_{[1,\infty)}$: all three aggregate predictable reductions are deterministic and jump at time $1$, the first two with jump size $(1-e^{-1})/2$ and the aggregate one with jump size $1-e^{-1}$. Nevertheless,
\be\label{eq:atomic.negative.gamma}
 \Gamma_t^{\{1,2\}}=
 \ind_{\{t\ge1\}}\log\frac{4e}{(1+e)^2}\le0,
 \qquad t\ge0,
\ee
with strict inequality for $t\ge1$. Thus deterministic aggregate reductions alone do not guarantee nondecreasing M\"obius coordinates, and continuity cannot be omitted from Proposition~\ref{prop:continuous.reductions.MO}. More is true: by Remark~\ref{rem:SML}, $(\tau^1,\tau^2)$ is not positively quadrant dependent, so its law is not that of any fatal shock model with independent shocks. Default atoms at predictable dates thus take the model outside the exogenous shock class. Continuity is sufficient but not necessary: in Section~\ref{subsect:multiv.Liu}, $\Gamma^B$ is nondecreasing even when $\delta_0$ jumps. Finally, the two drivers never jump together, so Proposition~\ref{Prop:tau.1.equal.tau.2} gives $\P(\tau^1=\tau^2<\infty)=0$.
\eex

\par\medskip

\brem \textbf{(Independent-block factorization)} \label{rem:independent.blocks}
A useful tractable specification of the general framework is obtained by writing, for each $j\in V$,
$$
K_t^j=X_t^j+\widetilde K_t^j,
$$
where $X^j$ is continuous, increasing, and adapted to $\mathbb F^X$, whereas $\widetilde K^j$ is increasing, c\`adl\`ag, and adapted to $\mathbb F^{\widetilde K}$, with $X^j_0=\widetilde K^j_0=0$. Assume that

$$
\mathcal F_\infty^X
\quad\text{and}\quad
\mathcal F_\infty^{\widetilde K}
$$
are independent, and set
$$
\mathcal F_t=\mathcal F_t^X\vee\mathcal F_t^{\widetilde K}.
$$
The unit-exponential thresholds $\Theta^1,\ldots,\Theta^n$ are mutually independent and independent of $\mathcal F_\infty^X\vee\mathcal F_\infty^{\widetilde K}$. With
$$
\tau^j:=\inf\{t\ge0:K_t^j\ge\Theta^j\},
$$
for $0 \le t\le\min_{1\le j\le n}t_j$ we have
$$
\begin{aligned}
&\P\left(\tau^1>t_1,\ldots,\tau^n>t_n\mid\mathcal F_t\right)\\
&\qquad=
\E\left[
\exp\left\{-\sum_{j=1}^n X_{t_j}^j\right\}
\middle|\mathcal F_t^X
\right]
\E\left[
\exp\left\{-\sum_{j=1}^n \widetilde K_{t_j}^j\right\}
\middle|\mathcal F_t^{\widetilde K}
\right].
\end{aligned}
$$
Indeed, the first exponential is $\mathcal F_\infty^X$-measurable and the second is $\mathcal F_\infty^{\widetilde K}$-measurable, so the independence of the terminal sigma-fields gives the conditional factorization with respect to $\mathcal F_t^X\vee\mathcal F_t^{\widetilde K}$.

Equivalently, given $\mathcal F_\infty$, $(\tau^1,\ldots,\tau^n)$ has the same law as when each $\tau^j$ is the minimum of a Cox-type time with continuous driver $X^j$ and a generalized Cox time driven by $\widetilde K^j$, with all $2n$ thresholds independent; the two blocks thus separate gradual deterioration from shock-driven default.

The first factor is the conditional joint Laplace transform of the continuous block. The second factor is a multivariate generalized-Cox survival term and is therefore covered by the general backward recursion of Theorem~\ref{Thm:Multivariate.Recursion}, applied with $\mathbb F^{\widetilde K}$ in place of $\mathbb F$. If, in addition, the equivalent conditions of Lemma~\ref{lem:tractable} hold for every aggregate $\widetilde K^J$, the second factor reduces to the
product representation of Theorem~\ref{Thm:Multivariate.Survival}.

This factorization does not enlarge the generalized-Cox class: $K^j=X^j+\widetilde K^j$ is already an admissible driver of the general framework. It identifies a tractable subclass in which two independent sources of the cumulative default mechanism can be separated at the level of conditional survival probabilities.

A specification of this type, with $X^j$ given by integrated CIR intensities (idiosyncratic, or loadings on common factors) and $\widetilde K$ the subordinator factor model of Section~\ref{subsect:multiv.Levy}, is proposed in \citet[Appendix~B]{sun2017online} to add stochastic spread dynamics to the Marshall--Olkin model; the factorization above gives its $\mathcal F_t$-conditional joint survival law in terms of the two blocks.

\erem

\subsection{Examples}\label{subsect:multiv.examples}
\subsubsection{Multivariate L\'evy Subordinators}\label{subsect:multiv.Levy}
Let $L=(L^1,\ldots,L^n)$ be a driftless $n$-dimensional L\'evy subordinator, L\'evy with respect to $\mathbb F$ in the sense of Section~\ref{Subsect:LevySubordinators}, with L\'evy measure $\nu$ on $\mathbb R_+^n\setminus\{0\}$ satisfying $\int\bigl(1\wedge(x_1+\cdots+x_n)\bigr)\,\nu(d\mathbf x)<\infty$, and joint Laplace exponent
\[
  \psi(\mathbf z)=\int_{\mathbb R_+^n\setminus\{0\}}\bigl(1-e^{-\langle\mathbf z,\mathbf x\rangle}\bigr)\,\nu(d\mathbf x).
\]
Let $K_t^i=z_iL_t^i$ with $z_i\ge0$, and for $J\subseteq V$ write $\mathbf z^J:=(z_i\ind_{\{i\in J\}})_{i=1}^n$, so that $K^J=\langle\mathbf z^J,L\rangle$. Since $(K^1,\ldots,K^n)$ has independent increments with respect to $\mathbb F$, Remark~\ref{rem:joint.jump.measure.criterion} applies: for every nonempty $J\subseteq V$, $R^J(s,u)=e^{-(u-s)\psi(\mathbf z^J)}$ and $S^J(t)=e^{-t\psi(\mathbf z^J)}$. Hence $\Gamma^B_t=\gamma^B t$, with
\begin{equation}\label{eq:multiv.Levy.gamma}
  \gamma^B=\int_{\mathbb R_+^n\setminus\{0\}}e^{-\sum_{r\notin B}z_rx_r}\prod_{i\in B}\bigl(1-e^{-z_ix_i}\bigr)\,\nu(d\mathbf x),
\end{equation}
which for $n=2$ are the rates $\gamma^1,\gamma^2,\gamma^{\{1,2\}}$ of Section~\ref{Subsect:LevySubordinators}. These are the shock rates of \citet[Theorem~1 and Lemma~2]{sun2017marshall}, and the unconditional law is the corresponding Marshall--Olkin law. Corollary~\ref{Cor:Multivariate.Joint.Distr} adds its $\mathcal F_t$-conditional form: for $0\le t\le\min_i t_i$,
\begin{equation}\label{eq:multiv.Levy.conditional}
  \P(\tau^1>t_1,\ldots,\tau^n>t_n\mid\F_t)
  =\exp\Bigl\{-\langle\mathbf z,L_t\rangle+t\psi(\mathbf z)-\sum_{\emptyset\neq B\subseteq V}\gamma^B\max_{i\in B}t_i\Bigr\},
\end{equation}
which extends \eqref{eq:Levy.biv}.

\paragraph{Linear factor model.}
Let $K^i_t=\sum_{k=1}^m a_{ik}L^k_t$, $a_{ik}\ge0$, where $(L^1,\dots,L^m)$ is an $m$-dimensional L\'evy process with respect to $\mathbb F$ whose coordinates are independent driftless one-dimensional subordinators, with L\'evy measures $\nu_k$ and Laplace exponents $\psi_k$ (for instance, $\mathbb F$ is the augmented filtration generated by $(L^1,\dots,L^m)$). Then $(K^1,\ldots,K^n)$ is a driftless $n$-dimensional subordinator, L\'evy with respect to $\mathbb F$, whose L\'evy measure is the sum over $k$ of the images of $\nu_k$ under $x\mapsto x\,(a_{1k},\ldots,a_{nk})$ \citep[Example~2]{sun2017marshall}; this is the case $\mathbf z=(1,\ldots,1)$ above. With $b^J_k:=\sum_{i\in J}a_{ik}$, we get $S^J(t)=\exp\bigl\{-t\sum_{k=1}^m\psi_k(b^J_k)\bigr\}$ and, by \eqref{eq:multiv.Levy.gamma},
\[
  \gamma^B=\sum_{k=1}^m\int_{(0,\infty)}e^{-x\sum_{r\notin B}a_{rk}}\prod_{i\in B}\bigl(1-e^{-xa_{ik}}\bigr)\,\nu_k(dx).
\]
The rates $\gamma^B$ are nonnegative by this formula; the factor structure determines which of them vanish. At a jump of $L^k$ of size $x$, each surviving name $i$ defaults, given $\mathcal F_\infty$, with probability $1-e^{-xa_{ik}}$, independently across names, so the names that default together form a random subset of $\{i:a_{ik}>0\}$. Accordingly, $\gamma^B>0$ if and only if $B\subseteq\{i:a_{ik}>0\}$ for some $k$ with $\nu_k\neq0$: in the shock representation of Remark~\ref{rem:MO.interpretation}, the names in $B$ can be the exact set of simultaneous defaults only if they all load on a common factor.

\subsubsection{The Construction of Liu (2020) for \texorpdfstring{$n$}{n} Names}\label{subsect:multiv.Liu}
Let $L$, $\nu$ and $\psi$ be as in Section~\ref{subsect:multiv.Levy}, and let $X$, $\vartheta$, $\phi_i$, $\delta_0$ and $\mathbb F$ be as in Section~\ref{subsect:Liu}, now with $n$ names:
\[
  K^i_t:=\phi_i(X)\,L^i_{\vartheta\delta_0(t)},\qquad i\in V;
\]
\citet{liu2020competing} considers $n=2$. Write $\boldsymbol\phi^J:=(\phi_i(X)\ind_{\{i\in J\}})_{i=1}^n$, $\boldsymbol\phi:=\boldsymbol\phi^V$, and let $\gamma^B$ be given by \eqref{eq:multiv.Levy.gamma} with $\mathbf z=\boldsymbol\phi$; these rates are $\mathcal F_0$-measurable. As in Section~\ref{subsect:Liu}, for every nonempty $J\subseteq V$ and $0\le s\le u$,
\[
  R^J(s,u)=\exp\bigl\{-\vartheta\bigl(\delta_0(u)-\delta_0(s)\bigr)\psi(\boldsymbol\phi^J)\bigr\},
\]
which is $\mathcal F_0$-measurable but, in general, random. Hence, by backward induction, each $\Psi_k$ in Theorem~\ref{Thm:Multivariate.Recursion} is $\mathcal F_0$-measurable, with $\Psi_{k-1}=R^{A_k}(s_{k-1},s_k)\,\Psi_k$, and $\E\bigl[e^{-K^V_{s_1}}\Psi_1\mid\F_t\bigr]=e^{-K^V_t}R^V(t,s_1)\,\Psi_1$. Since $\psi(\boldsymbol\phi^J)=\sum_{B\cap J\neq\emptyset}\gamma^B$, the computation in the proof of Corollary~\ref{Cor:Multivariate.Joint.Distr}, with $\vartheta\,\delta_0(\cdot)\,\psi(\boldsymbol\phi^J)$ in place of $-\log S^J$, gives, for $0\le t\le\min_i t_i$,
\begin{equation}\label{eq:multiv.Liu.conditional}
  \P(\tau^1>t_1,\ldots,\tau^n>t_n\mid\F_t)
  =e^{-K^V_t+\vartheta\delta_0(t)\psi(\boldsymbol\phi)}
   \exp\Bigl\{-\vartheta\sum_{\emptyset\neq B\subseteq V}\gamma^B\,\delta_0\bigl(\max_{i\in B}t_i\bigr)\Bigr\},
\end{equation}
which reduces to \eqref{eq:Liu.conditional} for $n=2$. Integrating out $\vartheta$ at $t=0$ gives $\P(\tau^1>t_1,\ldots,\tau^n>t_n\mid X)=\mathcal L_\vartheta\bigl(\sum_{\emptyset\neq B\subseteq V}\gamma^B\delta_0(\max_{i\in B}t_i)\bigr)$.

Excluding the trivial case $\delta_0\equiv0$, the conditions of Lemma~\ref{lem:tractable} hold for every $K^J$ if and only if each $\vartheta\,\psi(\boldsymbol\phi^J)$ is a.s.\ constant, for instance if $\vartheta\equiv1$ and $X$ is deterministic. Under this condition, for each nonempty $B\subseteq V$, $\Gamma^B_t=\vartheta\,\gamma^B\,\delta_0(t)$. The coefficients $\vartheta\gamma^B$ are a.s.\ constant, since they are obtained by M\"obius inversion of the a.s.\ constant family $\bigl(\vartheta\psi(\boldsymbol\phi^J)\bigr)_{\emptyset\neq J\subseteq V}$, and nonnegative by \eqref{eq:multiv.Levy.gamma}. (Since $\vartheta$ and $X$ are independent, outside the trivial case in which no default occurs, this in fact forces $\vartheta$ itself to be a.s.\ constant.) Consequently, $\Gamma^B$ is nondecreasing, also when $\delta_0$ jumps, and Remark~\ref{rem:MO.interpretation} gives a fatal shock representation with shock survival functions
$$
\P(E_B>t)=\exp\{-\vartheta\gamma^B\delta_0(t)\}.
$$

If $\delta_0$ is absolutely continuous, as in \citet{liu2020competing}, then under this condition $L_{\vartheta\delta_0(\cdot)}$ is an additive subordinator, and this is the time-inhomogeneous construction of \citet[Section~6]{sun2017marshall} and \citet[Appendix~A, Theorem~1]{sun2017online}. If $\delta_0$ jumps, the $S^J$ jump, defaults occur at the jump dates of $\delta_0$ with positive probability (Remark~\ref{rem:avoidance}(c)), and \eqref{FI} holds with discontinuous survival functions, outside the continuous settings of \citet{scherer2019exogenous} and \citet{sun2017online}.

\section*{Acknowledgments}
We thank Monique Jeanblanc for her careful reading, insightful comments, and valuable suggestions on an earlier version of this work, which have greatly improved the manuscript. We are also grateful to two anonymous referees for their comments on a previous version. All remaining errors are our own.

\printbibliography[heading=bibintoc]
\end{document}